\pdfoutput=1
\documentclass[aos]{imsart}

\RequirePackage{amsthm,amsmath,amsfonts,amssymb}
\RequirePackage[numbers,sort&compress]{natbib}
\RequirePackage[colorlinks,citecolor=blue,urlcolor=blue]{hyperref}
\RequirePackage{graphicx}
\usepackage{graphicx}
\usepackage{fix-cm}
\usepackage{subcaption}
\usepackage{booktabs}
\startlocaldefs
\theoremstyle{plain}

\newtheorem{theorem}{Theorem}
\newtheorem{lemma}{Lemma}
\newtheorem{remark}{Remark}
\newtheorem{assumption}{Assumption}
\newtheorem{corollary}{Corollary}
\theoremstyle{definition}

\newtheorem{example}{Example}

 \endlocaldefs

\begin{document}

\begin{frontmatter}
\title{Nonparametric inference for spot volatility in pure-jump semimartingales}
%\title{A sample article title with some additional note\thanksref{t1}}
\runtitle{Nonparametric inference for spot volatility}
%\thankstext{T1}{A sample additional note to the title.}

\begin{aug}
%%%%%%%%%%%%%%%%%%%%%%%%%%%%%%%%%%%%%%%%%%%%%%%
%% Only one address is permitted per author. %%
%% Only division, organization and e-mail is %%
%% included in the address.                  %%
%% Additional information can be included in %%
\author[A]{\fnms{Chengxin}~\snm{Yan} 
% %\thanks{[\textbf{Corresponding author indication should be put in the Acknowledgment section if necessary.}]}
 \ead[label=e1]{cxyan@smu.edu.sg}},
 \author[A]{\fnms{Dachuan}~\snm{Chen}\ead[label=e2]{dcchen@smu.edu.sg}
% %\orcid{0000-0000-0000-0000}
 }
 \and
 \author[A]{\fnms{Jia}~\snm{Li}\ead[label=e3]{jiali@smu.edu.sg}}
%%%%%%%%%%%%%%%%%%%%%%%%%%%%%%%%%%%%%%%%%%%%%%
%% Addresses                                %%
%%%%%%%%%%%%%%%%%%%%%%%%%%%%%%%%%%%%%%%%%%%%%%
\address[A]{School of Economics, 
 Singapore Management University
%University or Company Name \textbf{[Additional affiliations should be put in the Acknowledgments section]}
 \printead[presep={ ,\ }]{e1,e2,e3}}

%\address[B]{Department,
%University or Company Name \textbf{[Additional affiliations should be put in the Acknowledgments section]}\printead[presep={,\ }]{e2,e3}}
\end{aug}

\begin{abstract}
We provide a comprehensive analysis of spot volatility inference in pure-jump semimartingales under two asymptotic settings: fixed-$k$, where each local window uses a fixed number of observations, and large-$k$, where this number grows with sampling frequency. For both active- and possibly inactive-jump settings, we derive generally nonstandard, typically non-Gaussian limit distributions and establish valid inference, including when the jump-activity index is consistently estimated. Simulations show that fixed-$k$ asymptotics offer markedly better finite-sample accuracy, underscoring their practical advantage for nonparametric spot volatility inference.
\end{abstract}



\end{frontmatter}
%%%%%%%%%%%%%%%%%%%%%%%%%%%%%%%%%%%%%%%%%%%%%%
%% Please use \tableofcontents for articles %%
%% with 50 pages and more                   %%
%%%%%%%%%%%%%%%%%%%%%%%%%%%%%%%%%%%%%%%%%%%%%%
%\tableofcontents

\section{Introduction}
The statistical inference for the stochastic volatility of continuous-time It\^{o} semimartingales with continuous Brownian diffusion plus jump component have been extensively studied in high-frequency econometrics and statistics
 \cite{ait2014high,jacod2012discretization}. The existing literature mainly focuses on two types of volatility quantities. The first concerns semiparametric  inference for integrated volatility functionals over a nontrivial time interval (see, e.g., \cite{andersen2003modeling,
      barndorff2004econometric,
      mykland2009inference,
      todorov2012realized,
      jacod2013quarticity,
      li2016generalized,
      renault2017efficient,
      li2013volatility,
      li2017adaptive,
      li2019efficient,
      jacod2019estimating,
      li2021glivenko}). 
The second concerns nonparametric  inference for spot volatility at any given point in time (see, e.g.,
\cite{foster1996continuous,
      comte1998long,
      fan2008spot,
      kristensen2010nonparametric,
      jacod2021volatility,
      Bollerslev:2021,
      bollerslev2024optimalJOE,
      li2024reading}).

When the underlying processes are governed by pure-jump semimartingales, prior research has mainly focused on the semiparametric inference of integrated volatility functionals (see, e.g., 
\cite{woerner2007inference,
      todorov2011limit,
      todorov2012realizeda,
      todorov2013power}), 
activity index inference (see, e.g., \cite{todorov2015jump,
andersen2015fine,
hounyo2017local,
todorov2017testing,
kolokolov2022estimating}), 
and tests on the necessity of Brownian motion in price modeling (see, e.g.,
\cite{ait2010brownian,
      cont2011nonparametric,
      jing2012modeling,
      kong2015testing}). 
However, to the best of our knowledge, there is no literature on nonparametric  inference for spot volatility under  the pure-jump framework.

Motivated by the preceding research and to address the aforementioned theoretical gap, this paper investigates nonparametric inference for spot volatility in a continuous-time pure-jump semimartingale defined on a filtered probability space \((\Omega, \mathcal{F}, (\mathcal{F}_t)_{t \geq 0}, \mathbb{P})\), given by
\begin{align}
	\label{eq:Price-process}
	X_t = X_0 + \int_0^t b_s\,ds + \int_0^t \sigma_{s-}\,dZ_s,
\end{align}
where \(b\) and \(\sigma\) are stochastic processes with c\`adl\`ag sample paths, and \(Z\) is a symmetric \(\beta\)-stable process with activity  index \(\beta \in (0,2)\), satisfying for any \(t \geq s\) and \(u \in \mathbb{R}\),
\begin{align*}
	\mathbb{E}\!\left[e^{iu(Z_t-Z_s)}\mid \mathcal{F}_s\right] 
	= \exp\!\left(-\tfrac{1}{2} (t-s)|u|^{\beta}\right).
\end{align*}
This framework complements the extensive literature on nonparametric inference for spot volatility in Brownian diffusion models, corresponding to the special case \(\beta = 2\).

Two theoretical frameworks have been developed for spot volatility inference based on high-frequency data. The conventional approach assumes that the number of high-frequency observations, \(k\), within each local estimation window diverges to infinity at an appropriate rate as the sampling interval shrinks to zero. Analogous to standard kernel-based nonparametric analysis, this ``large-\(k\)'' theory ensures consistency of the nonparametric estimator and justifies Gaussian-based inference through the classical central limit theorem (CLT). 
In contrast, the ``fixed-\(k\)'' approach recently proposed by \cite{Bollerslev:2021} treats \(k\) as a fixed constant. Although the spot volatility estimator is no longer consistent when the number of observations is considered as fixed, their framework exploits the local Gaussianity of the Brownian diffusion model to conduct asymptotically valid inference. By eliminating the additional layer of asymptotic approximation associated with letting \(k \to \infty\), the fixed-\(k\) approach achieves more accurate finite-sample approximations.

Parallel to the prior literature, this paper develops both the fixed-$k$ and large-$k$ theories for nonparametric spot volatility inference in the pure-jump semimartingale model \eqref{eq:Price-process}. Once we deviate from the conventional Brownian setting (i.e., $\beta=2$), the analysis depends critically on whether the driving stable process $Z$ exhibits active or inactive jumps, corresponding to $\beta \in (1,2)$ and $\beta \leq 1$, respectively. While our main contribution concerns spot volatility inference for a given $\beta$, we also establish the validity of the proposed procedure when $\beta$ is unknown but can be consistently estimated using well-established methods (see, e.g., \cite{ait2009estimating,todorov2011limit,todorov2013power,todorov2015jump,kolokolov2022estimating,hounyo2017local}).

Under the fixed-$k$ framework with active jumps, we propose a local power-variation-type estimator with power index $p$. We show that, for any $p >0$, the (multiplicative) estimation error can be strongly approximated---or ``coupled''---by the sum of $k$ i.i.d.\ $p$-th absolute stable random variables. In the inactive case, we introduce a second-order difference-type estimator to remove the confounding effects of the drift component and establish that it can be coupled by the sum of i.i.d.\ $p$-th absolute stable random variables of the difference type. As in \cite{Bollerslev:2021}, the coupling theory facilitates asymptotically valid inference for spot volatility, including the construction of confidence intervals (CIs).

Under the large-$k$ framework with active jumps, we derive the asymptotic distribution of the $t$-statistics associated with the spot volatility estimators.  Specifically, for $p \in (0, \beta/2)$, the limiting distribution is Gaussian by the classical CLT. 
For $p=\beta/2$, the limiting distribution is also Gaussian, but at a faster (by a logarithmic factor) convergence rate.
For $p \in (\beta/2, \beta)$, the limiting distribution follows a totally right-skewed $\beta/p$-stable law according to the generalized CLT under the same asymptotic scheme.
In the case of inactive jumps, we propose a second-order difference-type estimator and establish analogous asymptotic results.

We conduct Monte Carlo experiments to compare the finite-sample performance of the fixed-$k$ and large-$k$ asymptotic theories for conducting inference. 
The results show that, in finite samples, the fixed-$k$ theory provides a considerably more accurate approximation to the distribution of the spot volatility estimators than the large-$k$ theory. We thus recommend the fixed-$k$ spot volatility inference procedure for practical applications, while noting that the large-$k$ theory is of independent theoretical interest, which is especially useful for further analysis on the semiparametric estimation and inference of general  integrated volatility functionals as studied by \cite{jacod2013quarticity} and \cite{li2017adaptive} among others.

The remainder of the paper is organized as follows. Section~\ref{sect:fixed-k} presents the fixed-$k$ theory, followed by the large-$k$ theory in  Section~\ref{sect:large-k}. Section~\ref{sect:simulations} reports simulation results. Section~\ref{sect:proofs} collects all proofs. Throughout the paper, $K$ denotes a generic positive constant whose value may change from line to line; if it depends on parameters such as $p$ or $T$, we write $K_{p,T}$. For a sequence of random variables $\{X_n\}$, we write $X_n = O_\mathbb{P}(1)$ if it is bounded in probability, that is, for any $\epsilon>0$ there exists $M<\infty$ such that $\sup_n \mathbb{P}(|X_n|>M)<\epsilon$. If $u_n$ is a sequence of positive numbers, then $X_n = O_\mathbb{P}(u_n)$ means $X_n/u_n = O_\mathbb{P}(1)$; similarly, $X_n = o_\mathbb{P}(1)$ if $X_n\to 0$ in probability, and $X_n = o_\mathbb{P}(u_n)$ if $X_n/u_n \to 0$ in probability. We use $\xrightarrow{\mathcal{L}}$ (resp. $\overset{\mathcal{L}}{=}$) to denote convergence (resp. equality) in distribution, and write $a_n \asymp b_n$ for two real sequences if there exists $K\ge 1$ such that $a_n/K \le b_n \le K a_n$. Finally, we assume the process $X$ is observed at discrete times $i\Delta_n$, $i=1,\ldots,n$, over the fixed time interval $[0,T]$ with $\Delta_n=T/n$, and denote the $i$-th increment of $X$ by $\Delta_i^n X \equiv X_{i\Delta_n}-X_{(i-1)\Delta_n}$. All limits are for $n\to\infty$.

%%%%%%% Section 2 %%%%%%%%%%
\section{Fixed-\texorpdfstring{$k$}{} inference theory}
\label{sect:fixed-k}
This section develops the fixed-$k$ theory for spot volatility inference under the pure-jump model \eqref{eq:Price-process}. Section \ref{sect:fixed-k-infeasible} first considers the ``infeasible'' case, where the activity index $\beta$ is assumed to be known. Section \ref{sect:fixed-k-feasible} then extends the analysis to the feasible setting in which $\beta$ is unknown but can be consistently estimated.

\subsection{Spot volatility inference with known $\beta$}
\label{sect:fixed-k-infeasible}
In this subsection, we consider the scenario in which the activity index $\beta$ is known. The main goal is to construct asymptotically valid CIs for the \emph{scaled} spot volatility, $\sigma_{n,t}$, defined as
\begin{align}
	\label{eq:scaled-volatility}
	\sigma_{n,t} \equiv \Delta_n^{1/\beta} \sigma_t.
\end{align}

We clarify that the scaling factor $\Delta_n^{1/\beta}$ is introduced to ensure that the object of interest is comparable across models with different values of $\beta$. To illustrate this, consider the benchmark Brownian diffusion model
\begin{align}
	\label{eq:Price-process-cont}
	X_t = X_0 + \int_0^t b_s\,ds + \int_0^t v_s\,dW_s,
\end{align}
where $W$ is a standard Brownian motion (i.e., the special case of $Z$ with $\beta = 2$) and $v$ is a stochastic volatility process. By the scaling property of stable processes (see, e.g., \cite{samoradnitsky:1994}), increments of model \eqref{eq:Price-process} scale as $\Delta_n^{1/\beta}$, whereas increments of model \eqref{eq:Price-process-cont} scale as $\Delta_n^{1/2}$. Therefore, it is not meaningful to directly compare the spot volatility $\sigma_t$ in \eqref{eq:Price-process} with $v_t$ in \eqref{eq:Price-process-cont}, as these coefficients are integrated with the driving stable processes for generating observed data, with distinct asymptotic scales. In contrast, the scaled volatility defined in \eqref{eq:scaled-volatility} normalizes the scaling difference, thereby allowing meaningful comparisons across models.

It is also worth clarifying from the outset that, although the scaled volatility $\sigma_{n,t}$ shrinks to zero as 
$n \to \infty$, the proposed CIs contract at the same rate, thereby delivering correct asymptotic coverage.

In what follows, we distinguish between two regimes of the pure-jump model depending on the jump activity of the driving stable process. When $\beta \in (1,2)$, the process exhibits active jumps, while for $\beta \leq 1$, it features {inactive jumps}. The construction of the estimators and the corresponding inference results differ notably between these two cases, and we therefore treat them separately in the subsequent discussion.

\subsubsection{The case with active jumps}
\label{sect:fixed-k-infeasible-1}
We first focus on the case with $\beta \in (1,2)$, which is closer to the benchmark Brownian diffusion model with $\beta = 2$ and thus serves as a natural point of comparison.\footnote{Pure-jump models with active jumps have found empirical support in the recent literature. For instance, \cite{kolokolov2022estimating} reports an average jump activity index of $\beta = 1.76$ based on 5-minute Bitcoin returns.} As shown below, this setting allows for a relatively straightforward inference procedure for $\sigma_{n,t}$, closely resembling that developed for the Brownian framework.

To construct the spot volatility estimator, we divide the sample into 
\(m_n\) non-overlapping blocks, each of which contains a fixed number, \(k\), of returns. Let \(\mathcal{I}_{n,j} \equiv \{(j-1)k+1, \ldots, jk\}\) denote the collection of indices in the $j$-th block, which spans the time interval \(\mathcal{T}_{n,j} \equiv [(j-1)k\Delta_n, jk\Delta_n)\) with length $k\Delta_n\to0$.
For the $j$-th block, we consider a spot estimator with the form
\begin{align}
\label{eq:modi-estimator}
    \hat{\sigma}_{n,j}(p) \equiv \frac{1}{k}\sum_{i \in \mathcal{I}_{n,j}} |\Delta_i^n X|^p, \,\, p>0.
\end{align}

To establish the asymptotic property for $\hat{\sigma}_{n,j}(p)$, we impose some mild regularity conditions as summarized in the following assumption.

\begin{assumption}
\label{assumption}
Suppose that there exists a sequence \((T_m)_{m \geq 1}\) of stopping times increasing to infinity and a sequence \((K_m)_{m \geq 1}\) of constants such that the following conditions hold for each \(m \geq 1\): (i) 
\(
|b_t| + |\sigma_t|+|\sigma_t|^{-1} \leq K_m\,\,  \text{for all } t \in [0, T_m]
\); (ii)
\(
\mathbb{E}[|\sigma_{t \wedge T_m} - \sigma_{s \wedge T_m}|^2] \leq K_m |t-s|^{2\kappa}\,\,  \text{for all } t, s \in [0, T]
\) and for some $\kappa>0$.
\end{assumption}

Condition (i) in Assumption~\ref{assumption} requires the drift and volatility processes to be locally bounded, which is a standard assumption in high-frequency econometrics and statistics (see, e.g., \cite{jacod2012discretization,ait2014high}).
Condition (ii) imposes local $\kappa$-H\"{o}lder continuity of the volatility process in the $L^2$ norm for some $\kappa>0$.
The H\"{o}lder continuity index $\kappa$ is allowed to be arbitrarily small.
When $\kappa = 1/2$, this condition can be verified if $\sigma$ is an It\^{o} semimartingale or a long-memory process driven by a fractional Brownian motion (see, e.g., \cite{comte1996long}).
When $\kappa \in (0,1/2)$, this condition also accommodates rough volatility models (see, e.g.,  \cite{gatheral2018volatility,el2019characteristic,fukasawa2022consistent,chong2024statistical}).

Theorem \ref{th:fixed-k-1} establishes the asymptotic property of $ \hat{\sigma}_{n,j} (p)$ under the fixed-$k$ framework.

\begin{theorem}
\label{th:fixed-k-1}
Suppose  Assumption \ref{assumption} holds and $\beta \in (1,2)$. Then it holds that for any $t \in \mathcal{T}_{n,j}$,
\[
\frac{\hat{\sigma}_{n,j}(p)}{\sigma_{n,t}^p}
- \frac{1}{k}\sum_{i \in \mathcal{I}_{n,j}} |Z_i|^p
= O_{\mathbb{P}}\!\left(\Delta_n^{\,(p\wedge 1)(\kappa \wedge (1 - \frac{1}{\beta}))}\right)
= o_{\mathbb{P}}(1),
\]
where $(Z_i)_{i \in \mathcal{I}_{n,j}}$ are i.i.d.\ $\beta$-stable random variables satisfying 
$\mathbb{E}[e^{iuZ_i}] = \exp(-|u|^{\beta}/2)$.
\end{theorem}

Theorem \ref{th:fixed-k-1} establishes that the ratio ${\hat{\sigma}_{n,j}(p)}/{\sigma_{n,t}^p}$ can be strongly approximated---or coupled---by the random variable $k^{-1}\sum_{i\in\mathcal{I}_{n,j}} |Z_i|^p$, whose distribution is known for any given $\beta$ and can therefore be used for inference on $\sigma_{n,t}$. This result parallels that of \cite{Bollerslev:2021} for the case $\beta = p = 2$, where the coupling variable follows a scaled chi-squared distribution with $k$ degrees of freedom. Moreover, the coupling error is of asymptotic order $\Delta_n^{\,(p\wedge 1)(\kappa \wedge (1 - 1/\beta))}$, implying that choosing $p < 1$ deteriorates the approximation, while selecting $p > 1$ does not improve the convergence rate. Accordingly, we recommend setting $p = 1$ in practical applications.

With $k$ fixed, the spot estimator cannot be regarded as consistent, because the coupling variable for the estimation error ratio remains non-degenerate as $\Delta_n \to 0$. Nevertheless, the coupling result enables the direct construction of asymptotically valid  CIs for $\sigma_{n,t}$. Specifically, for any $\alpha \in (0,1)$, we select constants $\bar{L}_{\alpha}(p;\beta)$ and $\bar{U}_{\alpha}(p;\beta)$ satisfying
\[
\mathbb{P}\!\left[\bar{L}_{\alpha}(p;\beta) \leq \frac{1}{\bar{S}_k(p)} \leq \bar{U}_{\alpha}(p;\beta)\right] = 1 - \alpha,
\]
where $\bar{S}_k(p) \equiv k^{-1} \sum_{i=1}^k |Z_i|^p$. Then, by Theorem \ref{th:fixed-k-1},
\[
\mathbb{P}\!\left[\bar{L}_{\alpha}(p;\beta) \leq \frac{\sigma_{n,t}^p}{\hat{\sigma}_{n,j}(p)} \leq \bar{U}_{\alpha}(p;\beta)\right] \to 1 - \alpha,
\]
which leads to the following corollary.

    \begin{corollary}
    \label{cor:fixed-k-1}
Under the conditions in {Theorem \ref{th:fixed-k-1}}, we obtain 
    \begin{align*}
        \lim_{n\rightarrow{\infty}} \mathbb{P}\left[(\bar{L}_{\alpha}(p;\beta){\hat{\sigma}_{n,j}}(p))^{1/p}\leq {\sigma_{n,t}}\leq (\bar{U}_{\alpha}(p;\beta){\hat{\sigma}_{n,j}}(p))^{1/p}\right] = 1-\alpha.
    \end{align*}  
    \end{corollary}

Corollary \ref{cor:fixed-k-1} states that, for the $j$-th block,  
\[
\bigl[(\bar{L}_{\alpha}(p;\beta)\,\hat{\sigma}_{n,j}(p))^{1/p},\; (\bar{U}_{\alpha}(p;\beta)\,\hat{\sigma}_{n,j}(p))^{1/p}\bigr]
\]
forms a CI for $\sigma_{n,t}$ with asymptotic coverage $1-\alpha$. The length of this interval can be minimized by choosing $[\bar{L}_{\alpha}(p;\beta),\,\bar{U}_{\alpha}(p;\beta)]$ as the $1-\alpha$ level highest density interval of the distribution of $1/\bar{S}_k(p)$.

The validity of this CI requires $\beta > 1$, which is a key condition ensuring that the coupling error in Theorem \ref{th:fixed-k-1} is asymptotically negligible. Intuitively, this condition guarantees that the movements driven by $Z$ are ``sufficiently active'' to dominate, on a microscopic scale, the contribution of the (bounded-variational) drift component, paralleling a similar phenomenon in the classical Brownian diffusion setting. In fact, under the no-drift case (i.e., $b \equiv 0$), Theorem \ref{th:fixed-k-1} and Corollary \ref{cor:fixed-k-1} extend to the full range $\beta \in (0,2)$, thereby allowing possibly inactive jumps. In the more realistic setting with drift, however, an alternative spot estimator is needed, and the asymptotic theory must be adjusted accordingly. We turn to this next.

\subsubsection{The case with possibly inactive jumps}
\label{sect:fixed-k-infeasible-2}
To handle the case with possibly inactive jumps in the presence of drift, we consider a second-order difference-type estimator for $\sigma_{n,t}$, defined as
\begin{align}
	\label{eq:modi-estimator-2}
	\tilde{\sigma}_{n,j}(p)
	\equiv 
	\frac{2}{k} 
	\sum_{i \in \tilde{\mathcal{I}}_{n,j}}
	\big|\Delta_{2i}^n X - \Delta_{2i-1}^n X\big|^{p},
	\,\, p>0,
\end{align}
where $\tilde{\mathcal{I}}_{n,j} = \{(j-1)k/2 + 1,\ldots,(j-1)k/2 + k/2\}$ and $k \ge 2$ is assumed to be even for notational convenience. Compared with the local power-variation estimator in \eqref{eq:modi-estimator}, the differencing term $\Delta_{2i}^n X - \Delta_{2i-1}^n X$ removes the contribution of the drift component, provided that the drift coefficient satisfies a smoothness condition stated in Assumption \ref{assumption-2} below.

\begin{assumption}
\label{assumption-2}
    Suppose that {\rm Assumption \ref{assumption}}
holds. Furthermore, 
$
\mathbb{E}[|b_{t\wedge T_m}-b_{s\wedge T_m}|^{2}] \leq K_m |t-s|^{2\tilde{\kappa}} 
$
for all $t,s\in[0,T]$ and for some $\tilde{\kappa}>0$.
\end{assumption}

Assumption~\ref{assumption-2} imposes a local $\tilde{\kappa}$-Hölder continuity condition on the drift process under the $L^{2}$ norm. This regularity requirement is satisfied by a wide class of processes, including It\^o semimartingales as well as processes exhibiting long-memory or rough-path behavior. Similar smoothness conditions are also standard in the literature on inference for integrated volatility in the presence of drift under pure-jump semimartingale models; see, for example,  \cite{todorov2013power}, who adopts a stronger assumption that the drift process follows a continuous-time It\^o semimartingale.

\begin{theorem}
\label{th:fixed-k-2}
Suppose  Assumption \ref{assumption-2} holds and $\beta \in (1/(1+\tilde{\kappa}),2)$. Then it holds that for any $t \in \mathcal{T}_{n,j}$,
\[
\frac{\tilde{\sigma}_{n,j}(p)}{\sigma_{n,t}^p} - \dfrac{2}{k} \sum_{i \in \tilde{\mathcal{I}}_{n,j}} |\tilde{Z}_i|^p = O_{\mathbb{P}}\left(\Delta_n^{(p\wedge1)(\kappa\wedge(1+\tilde{\kappa}-\frac{1}{\beta}))}\right)=o_{\mathbb{P}}(1),
\]
where $(\tilde{Z}_i)_{i \in \tilde{\mathcal{I}}_{n,j}}$ are i.i.d.\ $\beta$-stable random variables satisfying 
$\mathbb{E}[e^{iu\tilde{Z}_i}] = \exp(-|u|^{\beta})$.
\end{theorem}

Theorem \ref{th:fixed-k-2} shows that ${\tilde{\sigma}_{n,j}(p)}/{\sigma_{n,t}^p}$ can be strongly approximated---or coupled---by  
$
2k^{-1} \sum_{i \in \tilde{\mathcal{I}}_{n,j}} |\tilde{Z}_i|^{p}
$
under the general condition $\beta \in (1/(1+\tilde{\kappa}),\,2)$. In the ``best-case'' scenario with $\tilde{\kappa}=1$, meaning that $b$ is locally Lipschitz continuous, the approximation applies to the broad range $\beta \in (1/2,\,2)$. For the remaining range $\beta \in (0,\,1/2]$, the validity of such an approximation in the presence of drift remains an open problem. We note, however, that this regime is of limited empirical relevance for modeling typical high-frequency price movements. Hence, the practical implications of this theoretical limitation are minimal.

Although the procedure based on $\tilde{\sigma}_{n,j}(p)$ accommodates a broader range of $\beta$ than the approach in Section~\ref{sect:fixed-k-infeasible-1}, this robustness comes at a cost. First, it requires stronger smoothness conditions on the drift coefficient. Second, the coupling variable $2k^{-1}\sum_{i\in\tilde{\mathcal{I}}_{n,j}} |\tilde{Z}_i|^p$ is more dispersed than $k^{-1}\sum_{i\in \mathcal{I}_{n,j}} |Z_i|^p$ in Theorem \ref{th:fixed-k-1}, leading to a loss of efficiency and reflecting a robustness-efficiency trade-off. 

Equipped with Theorem \ref{th:fixed-k-2}, we can follow a similar strategy as in Section \ref{sect:fixed-k-infeasible-1} to conduct inference for $\sigma_{n,t}$, as shown in Corollary \ref{cor:fixed-k-2} below, where for any $\alpha\in(0,1)$, we choose 
$\tilde{L}_{\alpha}(p;\beta)$ and 
$\tilde{U}_{\alpha}(p;\beta)$ such that 
	$\mathbb{P}[\tilde{L}_{\alpha}(p;\beta)\leq {1}/{\tilde{S}_k(p)}\leq \tilde{U}_{\alpha}(p;\beta)] = 1-\alpha,$ where $  \tilde{S}_k(p)\equiv2k^{-1} \sum_{i\in\tilde{\mathcal{I}}_{n,j}} |\tilde{Z}_i|^p$.
    \begin{corollary}
    \label{cor:fixed-k-2}
Under the conditions in  {Theorem \ref{th:fixed-k-2}}, we obtain 
    \begin{align*}
        \lim_{n\rightarrow{\infty}} \mathbb{P}\left[(\tilde{L}_{\alpha}(p;\beta){\tilde{\sigma}_{n,j}}(p))^{1/p}\leq {\sigma_{n,t}}\leq (\tilde{U}_{\alpha}(p;\beta){\tilde{\sigma}_{n,j}}(p))^{1/p}\right] = 1-\alpha.
    \end{align*} 
    \end{corollary}

Corollary \ref{cor:fixed-k-2} states that
 $[(\tilde{L}_{\alpha}(p;\beta){\tilde{\sigma}_{n,j}(p)})^{1/p}, (\tilde{U}_{\alpha}(p;\beta){\tilde{\sigma}_{n,j}(p))^{1/p}}]$
 is a  CI for $\sigma_{n,t}$ with asymptotic level $1-\alpha$. 
 The length of this interval can be minimized by choosing $[\tilde{L}_{\alpha}(p;\beta),\,\tilde{U}_{\alpha}(p;\beta)]$ as the $1-\alpha$ level highest density interval of the distribution of $1/\tilde{S}_k(p)$.

\subsection{Spot volatility inference with unknown $\beta$}
\label{sect:fixed-k-feasible}

The inference procedures proposed in Section~\ref{sect:fixed-k-infeasible} rely on the assumption that $\beta$ is known; this mirrors the classical Brownian diffusion setting where $\beta$ is known to be 2. In this subsection, we show that these procedures remain asymptotically valid when $\beta$ is unknown but can be consistently estimated. The key step is to establish a technical continuity result, demonstrating that the relevant quantiles of the coupling variables vary continuously with the activity index. This ensures that replacing $\beta$ with a consistent estimator does not affect the asymptotic validity of the CIs.

Many consistent estimators for the activity index, $\beta$, have been proposed in the literature cited in the introduction. For concreteness, we review a few concrete examples below. 

\begin{example}
\label{ex:hat-beta-aos}
For some $\varpi>0$ and $\eta>0$, define
$
U(\varpi,\eta)^n_t
\equiv \sum_{i=1}^{[ t/\Delta_n ]} 
{1}_{\{\, | \Delta_i^n X | > \eta \Delta_n^\varpi \}}
$
as the number of increments whose magnitude exceeds $\eta \Delta_n^\varpi$. 
Fix $0<\eta<\eta'$ and $T>0$, Theorem 2 in \cite{ait2009estimating} shows that  $\beta$ can be consistently estimated by 
\begin{align}
\label{eq:beta-esti-1}
    \hat{\beta}_n
    \equiv \frac{\log\!\left({U(\varpi,\eta)^n_T}/{U(\varpi,\eta')^n_T}\right)}
    {\log(\eta'/\eta)}.
\end{align}
\end{example}

\begin{example}
\label{ex:hat-beta-aap}
Let
$
V_t(p,\Delta_n)\equiv\sum_{i=1}^{[t/\Delta_n]} |\Delta_i^nX|^p$ be the realized $p$-th power variation. Corollary 4.1 in \cite{todorov2011limit} shows that $\beta$ can be consistently estimated by 
\begin{equation}
\label{eq:beta-esti-2}
\hat{\beta}_n(p) \equiv
\dfrac{p\log(2)}{\log(2)+\log\left(V_T(p,2\Delta_n)\right) - \log\left(V_T(p,\Delta_n)\right)},\,\,p>0.
\end{equation}
\end{example}

\begin{example}
\label{ex:hat-beta-spa}
Let $V_t^1(p,\Delta_n)=\sum_{i=1}^{[t/\Delta_n]}|\Delta_i^n X-\Delta_{i-1}^n X|^p$ 
denote the realized $p$-th power variation based on second-order differences of $X$, 
and let $V_t^2(p,
\Delta_n)=\sum_{i=4}^{[t/\Delta_n]}|\Delta_i^n X-\Delta_{i-1}^n X+\Delta_{i-2}^n X-
\Delta_{i-3}^n X|^p$ be its temporally aggregated version. Define
\begin{equation}
\label{eq:beta-esti-3}
\hat{\beta}_n (p)\equiv
\dfrac{p\log(2)}{\log\left(V_t^2(p,\Delta_n)\right) - \log\left(V_t^1(p,\Delta_n)\right)}
{1}_{\{V_t^2(p,\Delta_n)\neq V_t^1(p,\Delta_n)\}},\,p>0.
\end{equation}
Corollary 2 in \cite{todorov2013power}
establishes that $\hat{\beta}_n(p)$ is a consistent estimator of $\beta$ for any $p \in (0,\beta)$.
\end{example}

Lemma \ref{le:quantile-continuity}, below, establishes the continuity of the quantiles of the coupling variables $\bar{S}_k(p)$ and $\tilde{S}_k(p)$ with respect to the activity index $\beta$.
To the best of our knowledge, this non-trivial technical result is new to the literature and will be useful in similar inferential settings.

\begin{lemma}
\label{le:quantile-continuity}
For each $\alpha \in (0,1)$, let $\bar{q}_{\alpha}(\beta)$ and $\tilde{q}_{\alpha}(\beta)$ denote the $\alpha$-quantiles of $\bar{S}_k(p)$ and $\tilde{S}_k(p)$, respectively.
Then both $\bar{q}_{\alpha}(\beta)$ and $\tilde{q}_{\alpha}(\beta)$ are continuous in $\beta \in (0,2)$.
\end{lemma}

Let $\bar{L}_{\alpha}(p;\hat{\beta}_n)$ and $\bar{U}_{\alpha}(p;\hat{\beta}_n)$ denote the  lower and upper bounds obtained by replacing $\beta$ with $\hat{\beta}_n$  in $\bar{L}_{\alpha}(p;{\beta})$ and $\bar{U}_{\alpha}(p;{\beta})$ for $\bar{S}_k(p)$. Similarly, let $\tilde{L}_{\alpha}(p;\hat{\beta}_n)$ and $\tilde{U}_{\alpha}(p;\hat{\beta}_n)$ denote the corresponding bounds for $\tilde{S}_k(p)$. Combining the consistency of $\hat{\beta}_n$ and Lemma \ref{le:quantile-continuity}, we obtain feasible CIs for $\sigma_{n,t}$ in the following corollary.

  \begin{corollary}
    \label{cor:fixed-k-3}
Let $\hat{\beta}_n$ be a consistent estimator for $\beta$. Under the conditions in {Theorem \ref{th:fixed-k-1}}, we obtain 
    \begin{align*}
        \lim_{n\rightarrow{\infty}} \mathbb{P}\left[(\bar{L}_{\alpha}(p;\hat{\beta}_n){\hat{\sigma}_{n,j}}(p))^{1/p}\leq \sigma_{n,t}\leq (\bar{U}_{\alpha}(p;\hat{\beta}_n){\hat{\sigma}_{n,j}}(p))^{1/p}\right] = 1-\alpha,
    \end{align*}  
     and under the conditions in {Theorem \ref{th:fixed-k-2}}, we obtain
    \begin{align*}
        \lim_{n\rightarrow{\infty}} \mathbb{P}\left[(\tilde{L}_{\alpha}(p;\hat{\beta}_n){\tilde{\sigma}_{n,j}}(p))^{1/p}\leq {\sigma_{n,t}}\leq  (\tilde{U}_{\alpha}(p;\hat{\beta}_n){\tilde{\sigma}_{n,j}}(p))^{1/p}\right] = 1-\alpha.
    \end{align*} 
    \end{corollary}

%%%%%%% Section 3 %%%%%%%%%%
\section{Large-\texorpdfstring{$k$}{} inference theory}
\label{sect:large-k}
We next develop an inference theory for spot volatility under the large-$k$ framework, where $k \to \infty$. To emphasize the dependence of $k$ on $n$, we write $k_n$ throughout this section. For clarity of exposition, we focus on the setting in which the activity index $\beta$ is known, as in Section~\ref{sect:fixed-k-infeasible}. The case with unknown but consistently estimable $\beta$ can be handled analogously, following the same argument as in Section~\ref{sect:fixed-k-feasible}, and is therefore omitted here for brevity.

\subsection{The case with active jumps}
\label{sect:large-k-infeasible-1}
As in Section \ref{sect:fixed-k-infeasible-1}, we focus on the active case $\beta\in(1,2)$ in this section.
For the $j$-th block, the scaled spot volatility estimator is defined as
\begin{align}
\label{eq:large-k-estimator}
    \hat{\sigma}_{n,j}(p,\beta) \equiv \frac{1}{c_{\beta}(p)k_{n}}\sum_{i \in \mathcal{I}_{n,j}} |\Delta_i^n X|^p, \,\, 0<p<\beta,
\end{align}
where the scaling factor $c_\beta(p)\equiv\mathbb{E}[|Z_1|^p]$ is introduced for normalization, with its explicit expression given by (see Page 163 in \cite{ken1999levy})
\begin{align}
\label{eq:constant-c_p}
    c_{\beta}(p)=\dfrac{2^{p-\frac{p}{\beta}}\Gamma\left(\dfrac{1+p}{2}\right)\Gamma\left(1-\dfrac{p}{\beta}\right)}{\sqrt{\pi}\Gamma\left(1-\dfrac{p}{2}\right)}.
\end{align}

Theorem \ref{th:large-k-1}, below, establishes some key strong approximation results for the spot volatility estimator and its smooth transformations when $p<\beta/2$. In this case, the coupling variable has finite second moment, which further admits an asymptotic Gaussian approximation as $k_n\to\infty$.
\begin{theorem}
    \label{th:large-k-1}
    Suppose that {\rm Assumption \ref{assumption}} holds and that $\beta\in(1,2)$ is given. 
    Assume also that 
    $0<p<\frac{\beta}{2}$ and $k_{n} \asymp\Delta_n^{-\gamma}$ for some $$0<\gamma <2p\left(1-\frac{1}{\beta}\right)\wedge\frac{\kappa}{\kappa+\frac{1}{2p}}.$$
{\color{black}Then, for the transformations $f(x) = \log(x)$ or $f(x)=x^r$ with $r>0$},
there exists $\varepsilon>0$ such that
\begin{align*}
   \dfrac{k_{n}^{1/2}(f(\hat{\sigma}_{n,j}(p,\beta)) - f(\sigma_{n,t}^p))}{ f'(\hat{\sigma}_{n,j}(p,\beta)) \hat{\sigma}_{n,j}(p,\beta)} - \Xi_{n,j}(p,\beta) = o_{\mathbb{P}}(\Delta_n^{\varepsilon}),
\end{align*}
where
\begin{align*}
	\label{eq:defi-U_nj}
	\Xi_{n,j}(p,\beta)
	\equiv \frac{1}{k_{n}^{1/2}} \sum_{i\in\mathcal{I}_{n,j}}
	\left(\frac{|\Delta_i^n Z|^p}{c_{\beta}(p)\,\Delta_n^{p/\beta}}-1\right).
\end{align*}
\end{theorem}

Theorem \ref{th:large-k-1} shows that the $t$-statistic associated with commonly used smooth transformations (logarithm and power) of the spot volatility estimator can be strongly approximated by $\Xi_{n,j}(p,\beta)$. Since this coupling variable is a normalized sum of i.i.d.\ random variables with zero mean and (non-random) finite variance $c_\beta(2p)/c_\beta(p)^2 - 1$, the classical CLT implies that
\[
\Xi_{n,j}(p,\beta)\xrightarrow{\mathcal{L}} 
\mathcal{N}\!\left(0,\frac{c_{\beta}(2p)}{c_{\beta}(p)^2}-1\right).
\]
Consequently, the $t$-statistic inherits the same asymptotic distribution, allowing for the construction of a CI for $f(\sigma_{n,t}^p)$ as stated in the following corollary.

\begin{corollary}
\label{cor:large-k}
Suppose that the conditions in {Theorem \ref{th:large-k-1}} hold. 
For any $\alpha\in(0,1)$, let $z_{1-\alpha/2}$ be the $1-\alpha/2$ quantile of  $\mathcal{N}(0,1)$. Then
\[
\lim_{n\rightarrow\infty}\mathbb{P}\!\left[
B^-_{n,j}(p,\beta)    \leq f(\sigma_{n,t}^p) 
    \leq B^+_{n,j}(p,\beta)\right] = 1-\alpha,
\]
where 
\[
B_{n,j}^{\pm}(p,\beta) \equiv f\!\big(\hat{\sigma}_{n,j}(p,\beta)\big)\;\pm\; 
k_n^{-1/2} z_{1-\alpha/2}{f'\!\big(\hat{\sigma}_{n,j}(p,\beta)\big)\,\hat{\sigma}_{n,j}(p,\beta)\,
\sqrt{\frac{c_{\beta}(2p)}{c_{\beta}(p)^2}-1}}.
\]
\end{corollary}

\begin{remark}
\label{Re:p=beta/2}
Theorem \ref{th:large-k-1} and Corollary \ref{cor:large-k} pertain to cases with $p<\beta/2$. When $p=\beta/2$, we can show that the t-statistic can be strongly approximated by
\begin{align*}
     \Xi_{n,j}(\beta)
    \;\equiv\;
    \frac{1}{\sqrt{k_{n}\log(k_{n})}} 
    \sum_{i\in\mathcal{I}_{n,j}}
    \left(
      \frac{|\Delta_i^n Z|^{\beta/2}}{c_{\beta}(\beta/2)\,\Delta_n^{1/2}}
      -1
    \right).
\end{align*}
Moreover, by applying a generalized CLT
(see Theorem~3.12 (b) in \cite{nolan2020univariate}), it follows that
\begin{align*}
  \Xi_{n,j}(\beta)
  \xrightarrow{\mathcal{L}} 
	\mathcal{N}\left(0, \frac{\Gamma(\beta)\sin(\pi\beta/2)}
  {2\pi\,c_{\beta}(\beta/2)^{2}}\right).
\end{align*}
CIs can be constructed accordingly using this asymptotic Gaussian approximation at the slightly faster convergence rate $\sqrt{k_n\log(k_n)}$. The technical details of this boundary case is omitted for brevity.
\end{remark}

We next turn to the case with $\beta/2<p<\beta$. When $p > \beta/2$, 
$\mathbb{E}[|\Delta_i^nZ|^{2p}]=\infty$, and so, the classical CLT fails in this scenario. 
To address this issue, we develop an alternative strong approximation as shown in Theorem \ref{th:large-k-1-stable} and derive the (non-Gaussian) asymptotic distribution in Lemma  \ref{le:lim-distri-S_nj} for the coupling variable.

\begin{theorem}
    \label{th:large-k-1-stable}
   Suppose that Assumption~\ref{assumption} holds and that $\beta\in(1,2)$ is given. 
Assume also that either (i) for $p\in(\beta/2,1)$ and $k_n\asymp\Delta_n^{-\gamma}$ such that 
\[
0<\gamma<
\frac{1-\frac{1}{\beta}}{\,\frac{1}{p}-\frac{1}{\beta}\,}
\wedge
\frac{\kappa}{\,\kappa+\frac{1}{p}-\frac{1}{\beta}\,},
\]
or (ii) for $p\in[1,\beta)$ and $k_n\asymp\Delta_n^{-\gamma}$ such that 
\[
0<\gamma<
\frac{\kappa}{\,\kappa+1-\frac{p}{\beta}\,}.
\] 
Then, {\color{black} for the transformations $f(x) = \log(x)$ or $f(x)=x^r$ with $r>0$}, 
there exists $\varepsilon>0$ such that
\begin{align*}
   \dfrac{k_{n}^{1-p/\beta}(f(\hat{\sigma}_{n,j}(p,\beta)) - f(\sigma_{n,t}^p))}{ f'(\hat{\sigma}_{n,j}(p,\beta)) \hat{\sigma}_{n,j}(p,\beta)} - {S}_{n,j}(p,\beta) = o_{\mathbb{P}}(\Delta_n^{\varepsilon}),
\end{align*}
where
\begin{align*}
    {S}_{n,j}(p,\beta) \equiv \frac{1}{k_{n}^{p/\beta}} 
   \sum_{i\in \mathcal{I}_{n,j}} 
    \left(\frac{|\Delta_i^nZ|^p}{c_{\beta}(p)\,\Delta_n^{p/\beta}}-1\right).
\end{align*}
\end{theorem}

Theorem \ref{th:large-k-1-stable} establishes strong approximations for the logarithmic and power transformations of the $t$-statistic associated with the spot volatility estimator in the non-classical setting. To characterize the limiting distribution of the coupling variable ${S}_{n,j}(p,\beta)$, we denote by $\mathcal{S}(\eta,\delta,\theta,\mu)$ the general stable distribution with activity index $\eta\in(0,2)$, skewness parameter $\delta\in[-1,1]$, scale parameter $\theta \ge 0$, and location parameter $\mu \in \mathbb{R}$. Specifically, if a random variable $\xi$ follows a stable distribution $\mathcal{S}(\eta,\delta,\theta,\mu)$, its characteristic function is given by (see Definition~1.5 in \cite{nolan2020univariate})
\begin{align*}
	\mathbb{E}[e^{iu\xi}] =
	\begin{cases}
		\exp\!\left\{-\theta^\eta |u|^\eta \left[1 - i\delta \, \left(\tan \frac{\pi \eta}{2}\right)\,(\operatorname{sign} u)\right] + i\mu u\right\}, 
		& \text{if } \eta \neq 1,
		\\[0.3cm]
		\exp\!\left\{-\theta |u| \left[1 + i\delta \frac{2}{\pi} (\operatorname{sign} u) \log |u|\right] + i\mu u\right\}, 
		& \text{if } \eta = 1,
	\end{cases}
\end{align*}
for all $u \in \mathbb{R}$.

\begin{lemma}
    \label{le:lim-distri-S_nj}
For any  $\beta\in(0,2)$ and $p\in(\beta/2,\beta)$,
$
{S}_{n,j}(p,\beta)\xrightarrow{\mathcal{L}} \mathcal{S}\left({\beta}/{p},1,C_{\beta}(p)^{-1},0\right),
$
 where 
\begin{align*}
    C_{\beta}(p) \equiv   c_{\beta}(p)\left(\dfrac{2\Gamma(\beta/p)\sin(\pi\beta/(2p))}{\Gamma(\beta)\sin(\pi\beta/2)}\right)^{p/\beta}.
\end{align*}
\end{lemma}

Lemma \ref{le:lim-distri-S_nj} shows that the limiting distribution of $S_{n,j}(p,\beta)$ is a totally right-skewed $\beta/p$-stable distribution,  
$
\mathcal{S}\!\left({\beta}/{p},\,1,\,C_{\beta}(p)^{-1},\,0\right),
$
with activity index $\beta/p \in (1,2)$ for $p \in (\beta/2,\beta)$. Combining this with Theorem \ref{th:large-k-1-stable}, we obtain
\begin{align*}
	\dfrac{k_{n}^{\,1-p/\beta}\,\bigl(f(\hat{\sigma}_{n,j}(p,\beta)) - f(\sigma_{n,t}^p)\bigr)}
	{ f'(\hat{\sigma}_{n,j}(p,\beta)) \,\hat{\sigma}_{n,j}(p,\beta)}
	\;\xrightarrow{\mathcal{L}}\;
	\mathcal{S}\!\left(\frac{\beta}{p},\,1,\,C_{\beta}(p)^{-1},\,0\right).
\end{align*}
This stable limit then yields a CI for $f(\sigma_{n,t}^p)$, as stated in the following corollary.

\begin{corollary}
\label{cor:large-k-stable}
Suppose that the conditions in {Theorem \ref{th:large-k-1-stable}} hold. For any $\alpha\in(0,1)$, let ${q}^L_{\alpha}(p,\beta)$ and $q^U_{\alpha}(p,\beta)$ be constants such that $\mathbb{P}[{q}^L_{\alpha}(p,\beta)\leq\zeta\leq {q}^U_{\alpha}(p,\beta)]=1-\alpha$
where $\zeta$ is a generic $\mathcal{S}({\beta}/{p},1,C_{\beta}(p)^{-1},0)$-distributed random variable. Then we obtain
\[
\lim_{n\rightarrow\infty}\mathbb{P}\!\left[
    B^{(s)-}_{n,j}(p,\beta)\leq f(\sigma_{n,t}^p) \leq B^{(s)+}_{n,j}(p,\beta)\right] = 1-\alpha,
\]
where 
\[
B_{n,j}^{(s)-}(p,\beta) \equiv f(\hat{\sigma}_{n,j}(p,\beta))-\frac{ f'(\hat{\sigma}_{n,j}(p,\beta)) \hat{\sigma}_{n,j}(p,\beta){q}^U_{\alpha}(p,\beta)}{k_{n}^{1-p/\beta}}, 
\]
\[
B_{n,j}^{(s)+}(p,\beta) \equiv f(\hat{\sigma}_{n,j}(p,\beta))-\frac{ f'(\hat{\sigma}_{n,j}(p,\beta)) \hat{\sigma}_{n,j}(p,\beta){q}^L_{\alpha}(p,\beta)}{k_{n}^{1-p/\beta}}. 
\]
\end{corollary}

\subsection{The case with possibly inactive jumps}
\label{sect:large-k-infeasible-2}

As in Section~\ref{sect:fixed-k-infeasible-2}, we employ a second-order difference-type estimator for $\sigma_{n,t}$ to accommodate the possibility of inactive jumps in the presence of drift. Specifically, the scaled spot volatility estimator for the $j$-th block is defined as
\begin{align}
	\label{eq:large-k-estimator-2}
	\tilde{\sigma}_{n,j}(p,\beta)
	\equiv 
	\frac{2}{\tilde{c}_{\beta}(p)\,k_{n}}
	\sum_{i \in \tilde{\mathcal{I}}_{n,j}}
	\bigl|\Delta_{2i}^n X - \Delta_{2i-1}^n X\bigr|^{p},\,\, 0 < p < \beta,
\end{align}
where $\tilde{\mathcal{I}}_{n,j} \equiv \{(j-1)k_n/2 + 1,\ldots,(j-1)k_n/2 + k_n/2\}$ and 
\[
\tilde{c}_{\beta}(p)
\equiv
\mathbb{E}\!\left[
\left|
\frac{\Delta_{2i}^n Z - \Delta_{2i-1}^n Z}{\Delta_n^{1/\beta}}
\right|^p
\right].
\]
By self-similarity and the independence of $\Delta_{2i}^n Z$ and $\Delta_{2i-1}^n Z$, the scaled increment 
$
\left({\Delta_{2i}^n Z - \Delta_{2i-1}^n Z}\right)/{\Delta_n^{1/\beta}}
$
is itself a symmetric $\beta$-stable random variable with distribution $\mathcal{S}(\beta,0,1,0)$.  
Analogous to the expression for $c_{\beta}(p)$ in~\eqref{eq:constant-c_p}, the constant $\tilde{c}_{\beta}(p)$ admits the explicit form (see page 163 in \cite{ken1999levy})
\begin{align}
	\label{eq:constant-tilde-c_p}
	\tilde{c}_{\beta}(p)
	=
	\frac{
		2^{p}\,
		\Gamma\!\left(\dfrac{1+p}{2}\right)
		\Gamma\!\left(1-\dfrac{p}{\beta}\right)
	}{
		\sqrt{\pi}\,
		\Gamma\!\left(1-\dfrac{p}{2}\right)
	}.
\end{align}

Theorem \ref{th:large-k-2}, stated below, provides a strong approximation for smooth transformation of the spot volatility estimator $\tilde{\sigma}_{n,j}(p,\beta)$ when $0 < p < \beta/2$. In this regime, the relevant coupling variable is a normalized sum of i.i.d.\ random variables with finite second moment, and therefore converges to a Gaussian limit by the classical CLT.

\begin{theorem}
	\label{th:large-k-2}
	Suppose that Assumption~\ref{assumption-2} holds and that 
	$\beta \in \bigl(1/(1+\tilde{\kappa}),\,2\bigr)$ is given.
	Assume also that $p\in(0,\beta/2)$ and 
	$k_{n} \asymp \Delta_n^{-\gamma}$ for some 
	\[
	0 < \gamma 
	< 
	\frac{\tilde{\kappa}+1 - \frac{1}{\beta}}{\tilde{\kappa} + \frac{1}{2p}}
	\;\wedge\;
	\frac{\kappa}{\kappa + \frac{1}{2p}}.
	\]
	Then, {\color{black} for the transformations $f(x) = \log(x)$ or $f(x)=x^r$ with $r>0$}, 
there exists $\varepsilon>0$ such that
	\[
	\dfrac{
		(k_{n}/2)^{1/2}\,
		\bigl(f(\tilde{\sigma}_{n,j}(p,\beta))
		- f(\sigma_{n,t}^{p})\bigr)
	}{
		f'(\tilde{\sigma}_{n,j}(p,\beta))\,
		\tilde{\sigma}_{n,j}(p,\beta)
	}
	\;-\;
	\tilde{\Xi}_{n,j}(p,\beta)
	\;=\;
	o_{\mathbb{P}}(\Delta_n^{\varepsilon}),
	\]
	where
	\[
	\tilde{\Xi}_{n,j}(p,\beta)
	\equiv 
	\frac{1}{(k_{n}/2)^{1/2}}
	\sum_{i \in \tilde{\mathcal{I}}_{n,j}}
	\left(
	\frac{|\Delta_{2i}^n Z - \Delta_{2i-1}^n Z|^{p}}
	{\tilde{c}_\beta(p)\,\Delta_n^{p/\beta}}
	- 1
	\right).
	\]
\end{theorem}

When $p < \beta/2$, the coupling variable $\tilde{\Xi}_{n,j}(p,\beta)$ is a normalized sum of i.i.d.\ random variables with finite second moment. By the classical CLT,
\[
\tilde{\Xi}_{n,j}(p,\beta)
\xrightarrow{\mathcal{L}}
\mathcal{N}\!\left(0,\frac{\tilde{c}_{\beta}(2p)}{\tilde{c}_{\beta}(p)^2}-1 \right).
\]
Combined with Theorem~\ref{th:large-k-2}, this yields
\begin{align*}
	\dfrac{
		(k_{n}/2)^{1/2}\,\bigl(f(\tilde{\sigma}_{n,j}(p,\beta)) - f(\sigma_{n,t}^{p})\bigr)
	}{
		f'(\tilde{\sigma}_{n,j}(p,\beta))\,\tilde{\sigma}_{n,j}(p,\beta)
	}
	\xrightarrow{\mathcal{L}}
	\mathcal{N}\!\left(0,\frac{\tilde{c}_{\beta}(2p)}{\tilde{c}_{\beta}(p)^2}-1 \right).
\end{align*}
This asymptotic Gaussian approximation allows for the construction of a CI for $f(\sigma_{n,t}^p)$, as stated in the following corollary.
\begin{corollary}
	\label{cor:large-k-2}
	Suppose that the conditions in {Theorem \ref{th:large-k-2}} hold. 
	For any $\alpha \in (0,1)$, let $z_{1-\alpha/2}$ denote the $1-\alpha/2$ quantile of $\mathcal{N}(0,1)$. Then
	\[
	\lim_{n\to\infty}
	\mathbb{P}\!\left[
	\tilde{B}_{n,j}^-(p,\beta)
	\;\le\;
	f(\sigma_{n,t}^p)
	\;\le\;
	\tilde{B}_{n,j}^+(p,\beta)
	\right]
	= 1-\alpha,
	\]
	where
	\[
	\tilde{B}_{n,j}^\pm(p,\beta)
	\equiv 
	f(\tilde{\sigma}_{n,j}(p,\beta))
	\pm
		(k_{n}/2)^{-1/2}\,
	z_{1-\alpha/2}
	f'(\tilde{\sigma}_{n,j}(p,\beta))\,\tilde{\sigma}_{n,j}(p,\beta)
	\sqrt{\tfrac{\tilde{c}_{\beta}(2p)}{\tilde{c}_{\beta}(p)^2}-1}.
	\]
\end{corollary}

\begin{remark} 
Theorem \ref{th:large-k-2} and Corollary \ref{cor:large-k-2} pertain to cases with $0<p<\beta/2$. 
Analogous to Remark \ref{Re:p=beta/2}, when $p=\beta/2$, we can show that the 
t-statistic can be strongly approximated by
\begin{align*}
     \tilde{\Xi}_{n,j}(\beta)
    \;\equiv\;
    \frac{1}{\sqrt{(k_{n}/2)\log(k_{n}/2)}} 
    \sum_{i\in\tilde{\mathcal{I}}_{n,j}}
    \left(
      \frac{|\Delta_{2i}^n Z-\Delta_{2i-1}^n Z|^{\beta/2}}{\tilde{c}_{\beta}(\beta/2)\,\Delta_n^{1/2}}
      -1
    \right).
\end{align*}
Moreover, by applying a generalized CLT
(see Theorem~3.12 (b) in \cite{nolan2020univariate}), it follows that
\begin{align*}
  \tilde{\Xi}_{n,j}(\beta)
  \xrightarrow{\mathcal{L}} 
\mathcal{N}\left(0,\frac{\Gamma(\beta)\sin(\pi\beta/2)}
         {\pi\,\tilde{c}_{\beta}(\beta/2)^{2}}\right).
\end{align*}
CIs can be constructed accordingly using this asymptotic Gaussian approximation at the rate $\sqrt{k_n\log(k_n)}$.
\end{remark}

We next turn to the case with $\beta/2<p<\beta$. When $p > \beta/2$, 
$\mathbb{E}[|\Delta_{2i}^nZ-\Delta_{2i-1}^nZ|^{2p}]=\infty$, and so, the classical CLT also fails in this scenario. 
To address this issue, we develop an alternative strong approximation as shown in Theorem \ref{th:large-k-2-stable} and derive the stable asymptotic distribution in Lemma  \ref{le:lim-distri-tilde-S_nj} for the coupling variable.

\begin{theorem}
    \label{th:large-k-2-stable}
    Suppose that {Assumption \ref{assumption-2}} holds and that $\beta\in(1/(1+\tilde{\kappa}),2)$ is given. 
Assume further that
 either (i) for $p\in(\beta/2,1\wedge\beta)$ and $k_{n} \asymp\Delta_n^{-\gamma}$ such that 
 $$0<\gamma <\frac{\tilde{\kappa}+1-\frac{1}{\beta}}{\tilde{\kappa}+\frac{1}{p}-\frac{1}{\beta}}\wedge\frac{\kappa}{\kappa+\frac{1}{p}-\frac{1}{\beta}},$$ 
or (ii) for $p\in[1,\beta)$ and $k_{n} \asymp\Delta_n^{-\gamma}$ such that $$0<\gamma <\frac{\kappa}{\kappa+1-\frac{p}{\beta}}.$$
Then, {\color{black} for the transformations $f(x) = \log(x)$ or $f(x)=x^r$ with $r>0$}, 
there exists $\varepsilon>0$ such that
\begin{align*}
   \dfrac{(k_{n}/2)^{1-p/\beta}(f(\tilde{\sigma}_{n,j}(p,\beta)) - f(\sigma_{n,t}^p))}{ f'(\tilde{\sigma}_{n,j}(p,\beta)) \tilde{\sigma}_{n,j}(p,\beta)} -\tilde{S}_{n,j}(p,\beta) = o_{\mathbb{P}}(\Delta_n^{\varepsilon}),
\end{align*}
where
\begin{align*}
\label{eq:tilde-stable-U_nj}
   \tilde{S}_{n,j}(p,\beta)\equiv \frac{1}{(k_{n}/2)^{p/\beta}} \sum_{i\in \tilde{\mathcal{I}}_{n,j}} \left(\frac{|\Delta_{2i}^nZ-\Delta_{2i-1}^nZ|^p}{\tilde{c}_{\beta}(p)\Delta_n^{p/\beta}}-1\right).
\end{align*}
\end{theorem}

\begin{lemma}
    \label{le:lim-distri-tilde-S_nj}
For any $\beta\in(0,2)$ and  $p\in(\beta/2,\beta)$,
$
\tilde{S}_{n,j}(p,\beta)\xrightarrow{\mathcal{L}} \mathcal{S}\left({\beta}/{p},1,\tilde{C}_{\beta}(p)^{-1},0\right),
$
where 
\[
\tilde{C}_\beta(p)
\equiv \tilde{c}_\beta(p)
\left(
\frac{\Gamma(\beta/p)\sin(\pi\beta/(2p))}
{\Gamma(\beta)\sin(\pi\beta/2)}
\right)^{p/\beta}.
\]
\end{lemma}
Lemma \ref{le:lim-distri-tilde-S_nj} shows that the limit distribution of $\tilde{S}_{n,j}(p,\beta)$ is a totally right-skewed $\beta/p$-stable distribution $\mathcal{S}({\beta}/{p},1,\tilde{C}_{\beta}(p)^{-1},0)$ with activity index $\beta/p\in(1,2)$ for $p\in (\beta/2,\beta)$.
Combing this with Theorem \ref{th:large-k-2-stable}, we obtain
\begin{align*}
\dfrac{(k_{n}/2)^{1-p/\beta}(f(\tilde{\sigma}_{n,j}(p,\beta)) - f(\sigma_{n,t}^p))}{ f'(\tilde{\sigma}_{n,j}(p,\beta)) \tilde{\sigma}_{n,j}(p,\beta)} \xrightarrow{\mathcal{L}} \mathcal{S}\left({\beta}/{p},1,\tilde{C}_{\beta}(p)^{-1},0\right).
\end{align*}
This stable limit then yields a CI for $f(\sigma_{n,t}^p)$, as stated in the following corollary.

\begin{corollary}
\label{cor:large-k-tilde-stable}
Suppose that the conditions in {Theorem \ref{th:large-k-2-stable}} hold. For any $\alpha\in(0,1)$, let $\tilde{q}^L_{\alpha}(p,\beta)$ and $\tilde{q}^U_{\alpha}(p,\beta)$ be constants such that $\mathbb{P}[\tilde{q}^L_{\alpha}(p,\beta)\leq\tilde{\zeta}\leq \tilde{q}^U_{\alpha}(p,\beta)]=1-\alpha$, where $\tilde{\zeta}$ is a generic 
$\mathcal{S}({\beta}/{p},1,\tilde{C}_{\beta}(p)^{-1},0)$-distributed random variable. Then we obtain
\[
\lim_{n\rightarrow\infty}\mathbb{P}\!\left[\tilde{B}_{n,j}^{(s)-}\leq 
    f(\tilde{\sigma}_{n,t}^p) \leq \tilde{B}_{n,j}^{(s)+}(p,\beta)\right] = 1-\alpha,
\]
where 
\begin{align*}
    \tilde{B}_{n,j}^{(s)-}(p,\beta) &=\frac{f(\tilde{\sigma}_{n,j}(p,\beta))-{ f'(\tilde{\sigma}_{n,j}(p,\beta)) \tilde{\sigma}_{n,j}(p,\beta)\tilde{q}^U_{\alpha}(p,\beta)}}{({k_{n}/2)^{1-p/\beta}}},
    \\
    \tilde{B}_{n,j}^{(s)+}(p,\beta) &=\dfrac{f(\tilde{\sigma}_{n,j}(p,\beta))-{ f'(\tilde{\sigma}_{n,j}(p,\beta)) \tilde{\sigma}_{n,j}(p,\beta)\tilde{q}^L_{\alpha}(p,\beta)}}{{(k_{n}/2)^{1-p/\beta}}}.
\end{align*}
\end{corollary}

%%%%%%% Section 4 %%%%%%%%%%

\section{Simulations}
\label{sect:simulations}
In this section, we conduct a Monte Carlo study to evaluate the finite-sample performance of the proposed inference on the scaled spot volatility $\sigma_{n,t}$ under both the fixed-$k$ and large-$k$ asymptotic frameworks. 
 The sample span is fixed at $T=1$ trading day, and 
the log-price process is simulated according to
\(
dX_t = \sigma_t \, dZ_t,
\)
where $Z$ is a stable process with activity index $\beta = 1.6$. To generate the volatility process $\sigma_t$, we follow the specification of \cite{bollerslev2011estimation}, modeling the spot variance as
\(
\sigma_t^2 = V_{1,t} + V_{2,t},
\)
with two independent volatility factors satisfying the following model:
\begin{align*}
    dV_{1,t} &= 0.0128(0.4068 - V_{1,t})\,dt + 0.0954\sqrt{V_{1,t}}\,dB_{1,t}, \\
    dV_{2,t} &= 0.6930(0.4068 - V_{2,t})\,dt + 0.7023\sqrt{V_{2,t}}\,dB_{2,t},
\end{align*}
where $B_{1}$ and $B_{2}$ are independent Brownian motions that are also independent of the stable process $Z$.  The 
$V_1$ volatility factor is highly persistent with a half-life of $2.5$ months, whereas the $V_2$ factor mean-reverts much faster, with a half-life of roughly one day. All continuous processes are simulated using the Euler scheme on $1$-second mesh, and the observed returns employed in the inference procedures are sampled at $\Delta_n=1$ minute intervals, yielding $390$ intraday returns over the trading day. All the numerical results given below are based on 100,000 Monte Carlo replications.

We compare the finite-sample distribution of the ratio estimation error at time $t = k\Delta_n$, namely $\hat{\sigma}_{n,t}(p)/\sigma_{n,t}^p$, with the corresponding limiting distributions under the fixed-$k$ and large-$k$ asymptotic theories. Recall that under the large-$k$ framework, the limit is Gaussian when $p < \beta/2$ and totally right-skewed stable when $p > \beta/2$. To illustrate both regimes, we consider $p = 0.6$ and $p = 1$, respectively. The fixed-$k$ limit distribution is expected to provide a more accurate finite-sample approximation, although the discrepancy should diminish as $k$ increases. To assess this behavior numerically, we examine a range of block sizes $k \in \{5, 15, 30, 60\}$.

Figure~\ref{fig:Simu_fix-larg-histo} reports the results, with the left and right panels corresponding to $p=0.6$ and $p=1$, respectively. The fixed-$k$ limit distributions track the finite-sample distributions remarkably well in both cases. By contrast, the large-$k$ limit distributions display a noticeable divergence from the empirical distributions---whether the theoretical limit is Gaussian or right-skewed stable. As predicted by the asymptotic theory, this discrepancy diminishes as $k$ increases, but it remains non-negligible even at $k = 60$. When $k$ is small, the large-$k$ limits yield particularly poor approximations.

In summary, we recommend using the fixed-$k$ asymptotic framework and its associated statistical procedure for spot volatility inference under the pure-jump setting. This conclusion is consistent with the findings of \cite{Bollerslev:2021} in the classical Brownian diffusion setting. By contrast, the large-$k$ approximation appears to require a substantially larger local sample size for the asymptotics to ``kick in'' adequately. This feature warrants caution in future work involving semiparametric or nonparametric inference on spot volatility, where reliance on large-$k$ asymptotics may lead to poor finite-sample performance.

\begin{figure}[htbp]
    \centering
    % First row
    \begin{subfigure}{0.5\textwidth}
        \centering
        \includegraphics[width=\linewidth,]{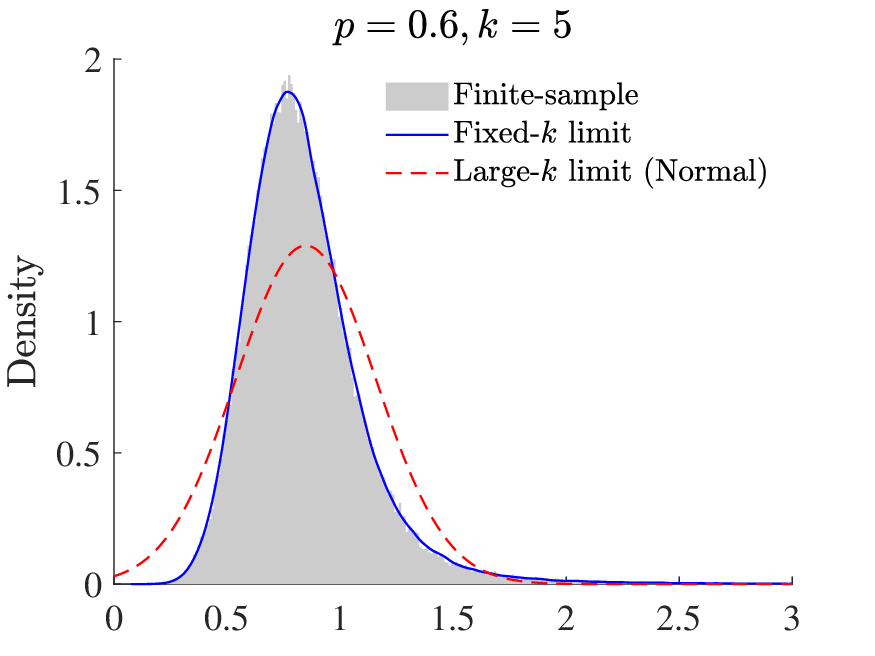}
    \end{subfigure}%
    \hfill
    \begin{subfigure}{0.5\textwidth}
        \centering
        \includegraphics[width=\linewidth]{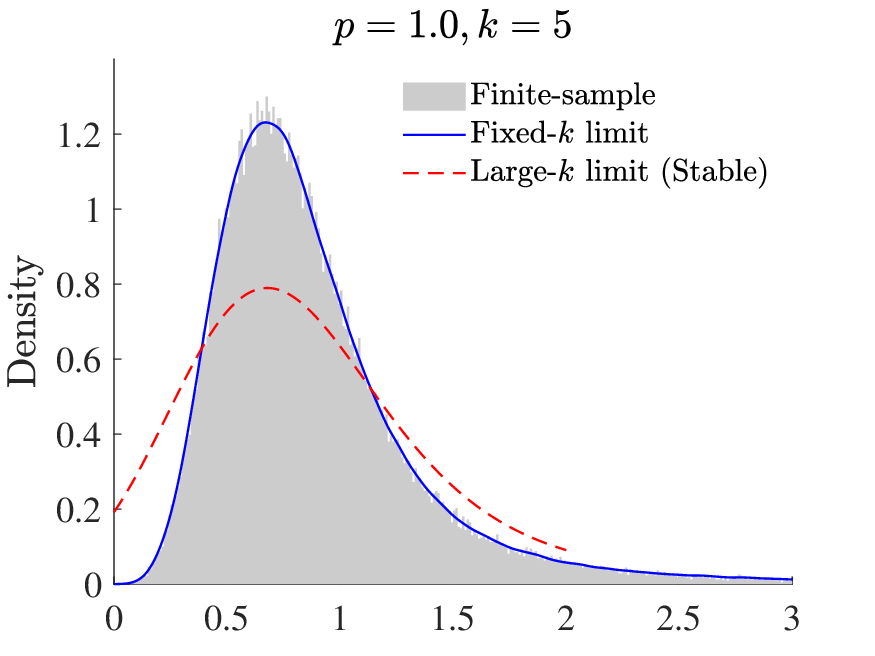}
    \end{subfigure}

    % Second row
    \begin{subfigure}{0.5\textwidth}
        \centering
        \includegraphics[width=\linewidth]{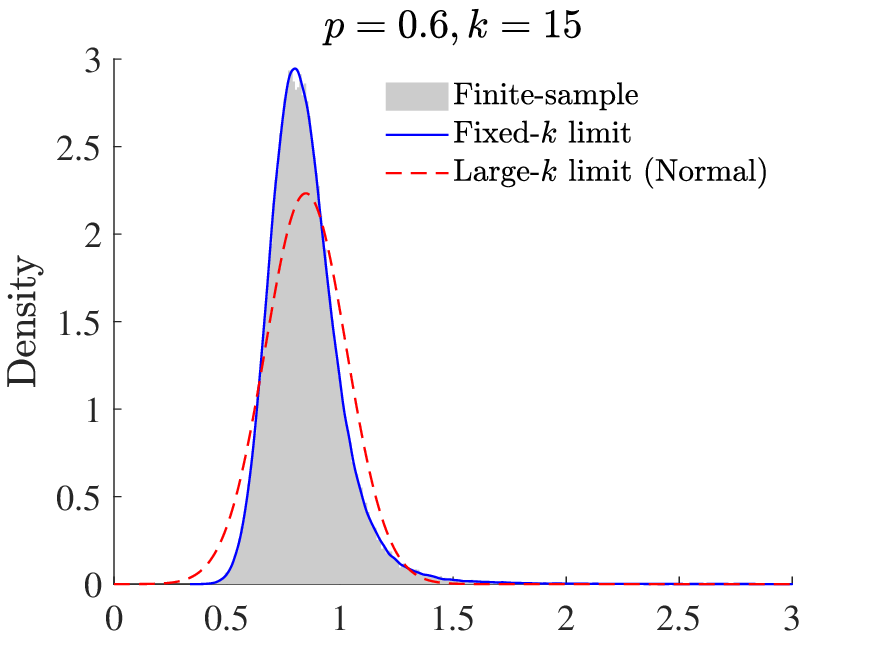}
    \end{subfigure}%
    \hfill
    \begin{subfigure}{0.5\textwidth}
        \centering
        \includegraphics[width=\linewidth]{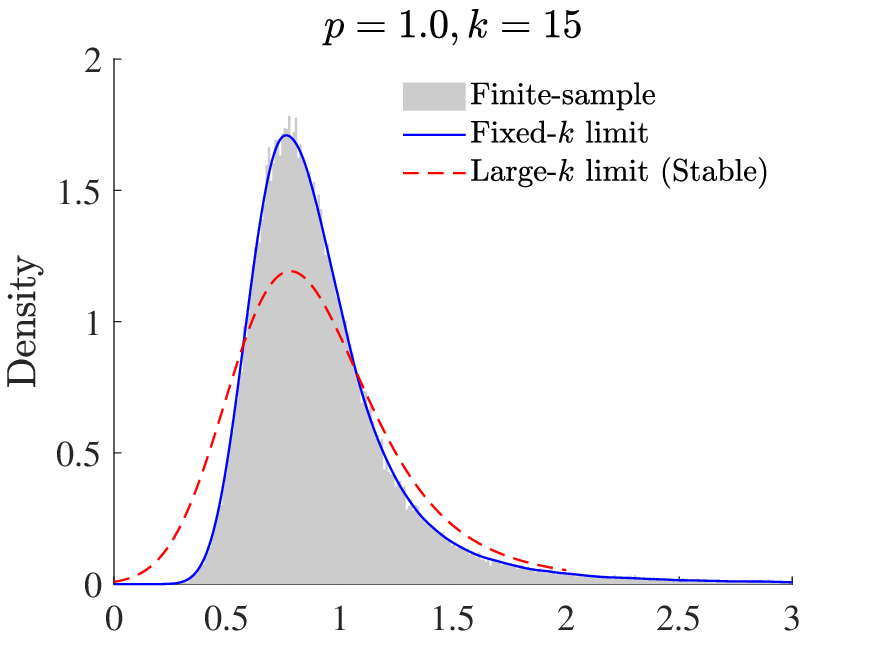}
    \end{subfigure}

    % Third row
    \begin{subfigure}{0.5\textwidth}
        \centering
        \includegraphics[width=\linewidth]{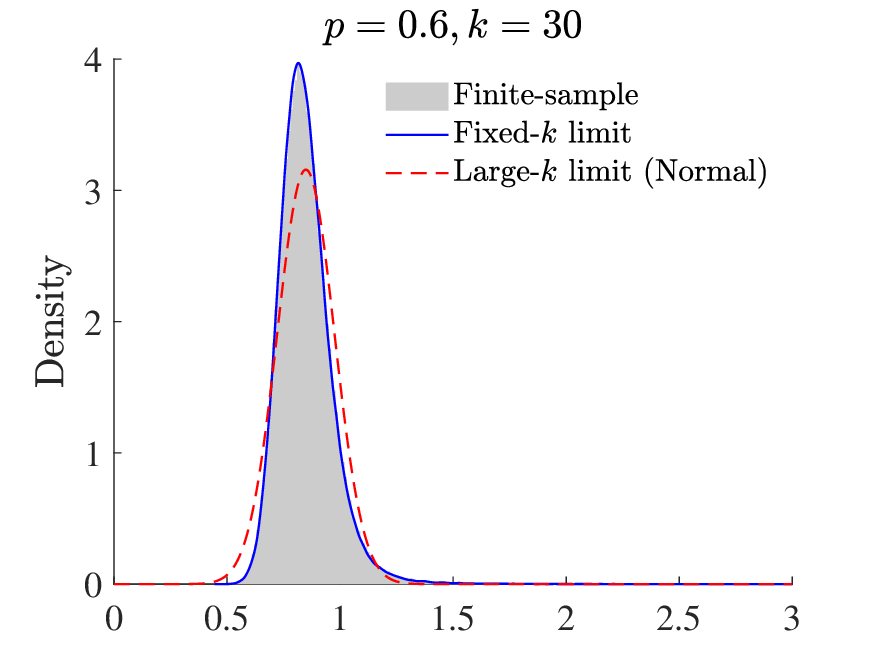}
    \end{subfigure}%
    \hfill
    \begin{subfigure}{0.5\textwidth}
        \centering
        \includegraphics[width=\linewidth]{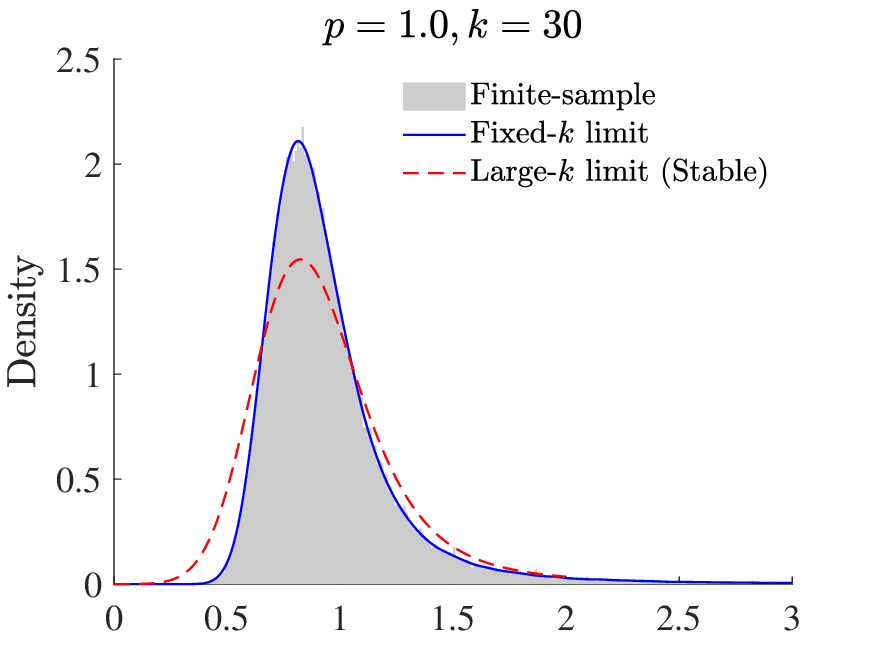}
    \end{subfigure}

    % Fourth row
    \begin{subfigure}{0.5\textwidth}
        \centering
        \includegraphics[width=\linewidth]{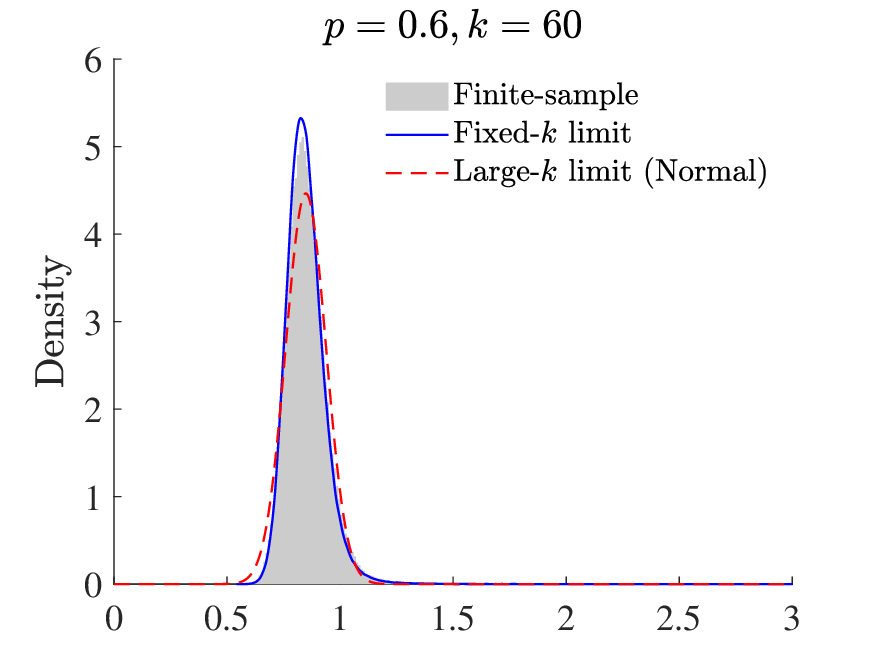}
    \end{subfigure}%
    \hfill
    \begin{subfigure}{0.5\textwidth}
        \centering
    \includegraphics[width=\linewidth]{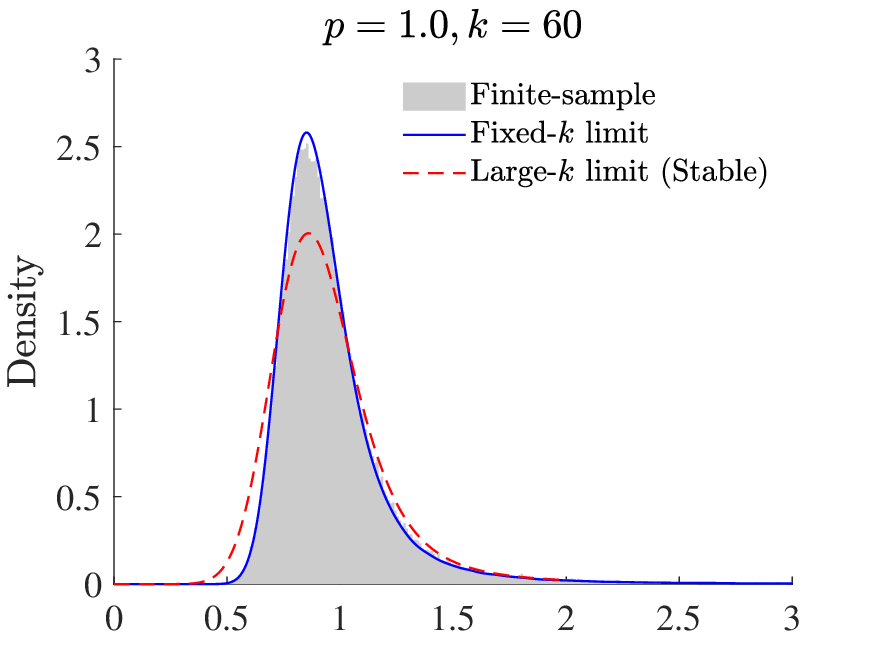}
    \end{subfigure}
\caption{Monte Carlo comparison between finite-sample distribution of $\hat{\sigma}_{n,j}(p)/\sigma_{n,t}^p$ and fixed-$k$ and large-$k$ limits.}
\label{fig:Simu_fix-larg-histo}
\end{figure}

%%%%%%%Section 5 %%%%%%%%%%
\section{Proofs}
\label{sect:proofs}
Throughout the proofs, for the $j$-th block, we use $t_{(n,j)}$ to denote the start point of the time interval $\mathcal{T}_{n,j}$,
that is, $t_{(n,j)}\equiv (j-1)k\Delta_n$ for fixed-$k$ setting and $t_{(n,j)}\equiv (j-1)k_{n}\Delta_n$ for large-$k$ setting. By a standard localization method, we can strengthen 
Assumptions \ref{assumption} and \ref{assumption-2} via setting $T_1=\infty$ without loss of generality 
(see section 4.4.1 in  \cite{jacod2012discretization}
for details on localization).

Before proving the main results in the paper, we first give a technique lemma which will be used in the sequel. For any $t\geq0$, define 
\begin{align*}
   Y_t\equiv \int_0^t \sigma_{s-}\,dZ_s.
\end{align*}
According to Assumption \ref{assumption}, the above stochastic integral is well-defined. Moreover, we have the 
following tail estimates (see, e.g., Theorem 3.5 in \cite{gine1983central} or
Theorem 2.1 in \cite{rosinski:1986}):
\begin{align}
\label{eq:tail-esti}
\mathbb{P}\left[|Y_t|^p>x\right] \leq 
\frac{K}{x^{\beta/p}}\mathbb{E}\left[\int_0^t|\sigma_s|^{\beta}\,ds\right],\,\,\forall\,\,p\in(0,\beta)\,\,\text{and}\,\, x>0.
\end{align}

The following lemma shows a moment estimate for process $Y$.
\begin{lemma}
\label{le:Mom-esti}
    Suppose that {Assumption \ref{assumption}} holds.
    Then, for any $t\geq0$ and $0<p<\beta$,
    \begin{align*}
\mathbb{E}\left[\left|Y_t\right|^p\right]  \leq 
    K_{p,\beta} \left( \mathbb{E}\left[\int_0^t|\sigma_s|^\beta\,ds\right]\right)^{p/\beta}.
    \end{align*}
\end{lemma}
\begin{proof}
For any $C>0$, we have
\begin{align*}
\mathbb{E}\left[\left|Y_t\right|^p\right]  
  &= \int_0^{\infty} \mathbb{P}\left[\left|\int_{0}^{t}\sigma_{s-}\,dZ_s\right|^p>x\right] \,dx
  \\
  &= \int_0^{C} \mathbb{P}\left[\left|\int_{0}^{t}\sigma_{s-}\,dZ_s\right|^p>x\right] \,dx
  + \int_C^{\infty} \mathbb{P}\left[\left|\int_{0}^{t}\sigma_{s-}\,dZ_s\right|^p>x\right] \,dx
  \\
  &\leq C + K_{p,\beta}C^{1-\beta/p}\mathbb{E}\left[\int_0^t|\sigma_s|^{\beta}\,ds\right],
\end{align*}
 where the last inequality follows by (\ref{eq:tail-esti}).
Define
\begin{align*}
    g(C) \equiv C + K_{p,\beta}C^{1-\beta/p}\mathbb{E}\left[\int_0^t|\sigma_s|^{\beta}\,ds\right], \,\,C>0.
\end{align*}
Since $p\in(0,\beta)$, by letting $g'(C) = 0$, 
the minimal 
value of $g$ is given by
\begin{align*}
      K_{p,\beta}^{p/\beta}\left(\frac{\beta}{\beta-p}\right) \left(\frac{\beta-p}{p}\right)^{p/\beta}\left( \mathbb{E}\left[\int_0^t|\sigma_s|^\beta\,ds\right]\right)^{p/\beta},
\end{align*}
which implies that
\begin{align*}
\mathbb{E}\left[\left|Y_t\right|^p\right]  \leq 
    K_{p,\beta} \left( \mathbb{E}\left[\int_0^t|\sigma_s|^\beta\,ds\right]\right)^{p/\beta}.
\end{align*}
The proof is finished.
\end{proof}

%%%%%%%%%%%%%%%%%%%%%%%%%%%%%%%%%%%%%%%%%%
%%%%%%%%% Proofs in Section 2 %%%%%%%%%%%%
%%%%%%%%%%%%%%%%%%%%%%%%%%%%%%%%%%%%%%%%%%

\subsection{Proofs in {\rm Section \ref{sect:fixed-k}}}

%%%%%%%%%%%%%%%%%%%%%%%%%%%%%%%%%%%%%%%%%%
%%%%%%%%% Proof of Theorem 1 %%%%%%%%%%%%%
%%%%%%%%%%%%%%%%%%%%%%%%%%%%%%%%%%%%%%%%%%

\begin{proof}[Proof of Theorem \ref{th:fixed-k-1}]
For the $j$-th block, define
    \begin{align*}
        \Theta_{n,j}(p,\beta)\equiv\dfrac{1}{\Delta_n^{p/\beta}}\sum_{i\in\mathcal{I}_{n,j}}|\Delta_{i}^nZ|^p,\,\,p>0.
    \end{align*}
By the self-similarity property $\Delta_i^nZ\overset{\mathcal{L}}=\Delta_n^{1/\beta}Z_1$,
we have
\begin{align}
\label{eq:same-law-Snj}
\Theta_{n,j}(p,\beta)\overset{\mathcal{L}}=\sum_{i\in\mathcal{I}_{n,j}}|Z_i|^p,
\end{align}
where $(Z_i)_{i\in\mathcal{I}_{n,j}}$ are i.i.d. symmetric $\beta$-stable 
variables satisfying 
$\mathbb{E}[e^{iuZ_i}]\equiv e^{-|u|^{\beta}/2}$. Moreover, 
we have the following 
decomposition:
\begin{align}
\label{eq:decomposition}
&k \hat{\sigma}_{n,j}(p) - \sigma_{n,t}^{p} \Theta_{n,j}(p,\beta) 
\nonumber\\
&= k\hat{\sigma}_{n,j}(p) - \Delta_n^{p/\beta}\sigma_{t_{(n,j)}}^p \Theta_{n,j}(p,\beta)+(\Delta_n^{p/\beta}\sigma_{t_{(n,j)}}^p - \sigma_{n,t}^{p}) \Theta_{n,j}(p,\beta).
\end{align}

Note that 
\begin{align*}
   \mathbb{E}[|(\Delta_n^{p/\beta}\sigma_{t_{(n,j)}}^p - \sigma_{n,t}^{p}) \Theta_{n,j}(p,\beta)|] &\leq  \Delta_n^{p/\beta}  \mathbb{E}[|\sigma_{t_{(n,j)}}^p-\sigma_t^p|] 
   \nonumber\\
   &\leq K_p \Delta_n^{p/\beta} \mathbb{E}[|\sigma_{t_{(n,j)}}-\sigma_t|^{1\wedge p}]
   \nonumber\\
   &\le K_p \Delta_n^{p/\beta+\kappa(1\wedge p)},
\end{align*}
where the first inequality follows by $\sigma_{n,t}=\Delta_n^{1/\beta}\sigma_t$ and $\Theta_{n,j}(p,\beta)=O_{\mathbb{P}}(1)$, 
the second inequality follows from the inequality 
$|a^p - b^p| \le K_p |a - b|^{1\wedge p}$, for all $p>0$ and $a,b>0$,
and the last inequality follows from
condition (ii) in Assumption \ref{assumption}. Then we have
\begin{align}
\label{eq:second-term}
  |(\Delta_n^{p/\beta}\sigma_{t_{(n,j)}}^p - \sigma_{n,t}^{p}) \Theta_{n,j}(p,\beta)|=  O_{\mathbb{P}}(\Delta_n^{p/\beta+\kappa(1\wedge p)}).
\end{align}

By \eqref{eq:Price-process}, we have
\begin{align*}
  \Delta_i^nX-\sigma_{t_{(n,j)}}\Delta_i^nZ =   
  \int_{(i-1)\Delta_n}^{i\Delta_n}b_s\,ds+\int_{(i-1)\Delta_n}^{i\Delta_n}(\sigma_{s-}-\sigma_{t_{(n,j)}})\,dZ_s.
\end{align*}
Then 
\begin{align*}
% \label{eq:estimation-p-th-difference}
    \mathbb{E}[|\Delta_i^nX-\sigma_{t_{(n,j)}}\Delta_i^nZ|]
    % \nonumber\\
    &\leq K\mathbb{E}\left[\left|\int_{(i-1)\Delta_n}^{i\Delta_n}b_s\,ds\right|\right]
    +K\mathbb{E}\left[\left|\int_{(i-1)\Delta_n}^{i\Delta_n}(\sigma_{s-}-\sigma_{t_{(n,j)}})\,dZ_s\right|\right]
    \nonumber\\
    &\leq K_p\Delta_n+ K_{\beta} \left(\int_{(i-1)\Delta_n}^{i\Delta_n}\mathbb{E}\left[|\sigma_{s}-\sigma_{t_{(n,j)}}|^{\beta}\right]\,ds\right)^{1/\beta}
    \nonumber\\
    &\leq K \Delta_n+ K_{\beta} \Delta_n^{\kappa+1/\beta},
\end{align*}
where the first inequality follows from the triangle inequality, 
the second inequality follows from condition~(i) in Assumption~\ref{assumption} 
and Lemma~\ref{le:Mom-esti}, 
and the third inequality follows from the Cauchy-Schwarz inequality 
together with condition~(ii) in Assumption~\ref{assumption}. Hence, we obtain
\begin{align}
\label{eq:key-estimate-fixed-k}
    |\Delta_i^nX-\sigma_{t_{(n,j)}}\Delta_i^nZ| = O_{\mathbb{P}}(\Delta_n^{1\wedge(\kappa+1/\beta)}).
\end{align}

For $p\in(0,1]$,  
we have
$$
     ||\Delta_i^nX|^p-|\sigma_{t_{(n,j)}}\Delta_i^nZ|^p|
     \leq K_p|\Delta_i^nX-\sigma_{t_{(n,j)}}\Delta_i^nZ|^p.
$$   
And so,
$$
    ||\Delta_i^nX|^p-|\sigma_{t_{(n,j)}}\Delta_i^nZ|^p|  = O_{\mathbb{P}}(\Delta_n^{p(1\wedge(\kappa+1/\beta))}).$$
    Then it holds by
the triangle inequality that
\begin{align}
\label{eq:first-term-smaller1}
   | k\hat{\sigma}_{n,j}(p) - \Delta_n^{p/\beta}\sigma_{t_{(n,j)}}^p \Theta_{n,j}(p,\beta)| = O_{\mathbb{P}}(\Delta_n^{p(1\wedge(\kappa+1/\beta))}).
\end{align}
Combing \eqref{eq:decomposition}, \eqref{eq:second-term} and \eqref{eq:first-term-smaller1}, we have
\begin{align*}
   k \hat{\sigma}_{n,j}(p) - \sigma_{n,t}^{p} \Theta_{n,j}(p,\beta)  = O_{\mathbb{P}}(\Delta_n^{p(1\wedge(\kappa+1/\beta))}).
\end{align*}
Since $\sigma_{n,t}^{p}=O_\mathbb{P}(\Delta_n^{p/\beta})$, we obtain
\begin{align}
\label{eq:case-p-smaller-1}
     \dfrac{\hat{\sigma}_{n,j}(p)}{\sigma_{n,t}^{p}} - \dfrac{1}{k} \Theta_{n,j}(p,\beta)  = O_{\mathbb{P}}(\Delta_n^{p(\kappa\wedge(1-1/\beta))}).
\end{align}

For $p>1$, by  
$|a^p - b^p| \le p\,(a^{p-1} + b^{p-1})|a - b|,\,\forall a,b \ge 0$, we obtain
\begin{align*}
  ||\Delta_i^nX|^p-|\sigma_{t_{(n,j)}}\Delta_i^nZ|^p|
     \leq K_p((|\Delta_i^nX|^{p-1}+|\sigma_{t_{(n,j)}}\Delta_i^nZ|^{p-1})|\Delta_i^nX-\sigma_{t_{(n,j)}}\Delta_i^nZ|).  
\end{align*}
Since $|\Delta_i^nX|=O_{\mathbb{P}}(\Delta_n^{1/\beta})$ and $|\Delta_i^nZ|=O_{\mathbb{P}}(\Delta_n^{1/\beta})$, it holds by 
\eqref{eq:key-estimate-fixed-k} that
\begin{align*}
    ||\Delta_i^nX|^p-|\sigma_{t_{(n,j)}}\Delta_i^nZ|^p| = O_{\mathbb{P}}(\Delta_n^{1\wedge(\kappa+1/\beta)+(p-1)/\beta}).
\end{align*}
And so, the triangle inequality implies that
\begin{align}
\label{eq:first-term-bigger1}
   | k\hat{\sigma}_{n,j}(p) - \Delta_n^{p/\beta}\sigma_{t_{(n,j)}}^p \Theta_{n,j}(p,\beta)| = O_{\mathbb{P}}(\Delta_n^{1\wedge(\kappa+1/\beta)+(p-1)/\beta}).
\end{align}
Combing \eqref{eq:decomposition}, \eqref{eq:second-term} and \eqref{eq:first-term-bigger1}, we have
\begin{align*}
   k \hat{\sigma}_{n,j}(p) - \sigma_{n,t}^{p} \Theta_{n,j}(p,\beta)  =  O_{\mathbb{P}}(\Delta_n^{1\wedge(\kappa+1/\beta)+(p-1)/\beta}).
\end{align*}
Since $\sigma_{n,t}^{p}=O_\mathbb{P}(\Delta_n^{p/\beta})$, we have
\begin{align}
\label{eq:case-p-bigger-1}
    \dfrac{\hat{\sigma}_{n,j}(p)}{\sigma_{n,t}^{p}} - \dfrac{1}{k} \Theta_{n,j}(p,\beta)  = O_{\mathbb{P}}(\Delta_n^{\kappa\wedge(1-1/\beta)}).
\end{align}
The proof is finished by combing \eqref{eq:same-law-Snj}, \eqref{eq:case-p-smaller-1} and \eqref{eq:case-p-bigger-1}.
\end{proof}

%%%%%%%%%%%%%%%%%%%%%%%%%%%%%%%%%%%%%%%%%%
%%%%%%%%% Proof of Theorem 2 %%%%%%%%%%%%%
%%%%%%%%%%%%%%%%%%%%%%%%%%%%%%%%%%%%%%%%%%

\begin{proof}[Proof of Theorem \ref{th:fixed-k-2}]
By (\ref{eq:Price-process}), we have
 \begin{align*}
     &\Delta_{2i}^nX-\Delta_{2i-1}^nX - \sigma_{t_{(n,j)}}(\Delta_{2i}^nZ-\Delta_{2i-1}^nZ)
     \\
     &=\int_{(2i-1)\Delta_n}^{2i\Delta_n} (b_s-b_{t_{(n,j)}})\,ds
     -\int_{(2i-2)\Delta_n}^{(2i-1)\Delta_n} (b_{s}-b_{t_{(n,j)}})\,ds
     \\
     &\quad+ 
     \int_{(2i-1)\Delta_n}^{2i\Delta_n} (\sigma_{s-}-\sigma_{t_{(n,j)}})\,dZ_s-\int_{(2i-2)\Delta_n}^{(2i-1)\Delta_n} (\sigma_{s-}-\sigma_{t_{(n,j)}})\,dZ_s.
 \end{align*}
By the triangle inequality, we have that for $p>0$,
 \begin{align*}
  &\mathbb{E}[|\Delta_{2i}^nX-\Delta_{2i-1}^nX - \sigma_{t_{(n,j)}}(\Delta_{2i}^nZ-\Delta_{2i-1}^nZ)|^p]
  \\
  &\leq K_p \mathbb{E}\left[\left|\int_{(2i-1)\Delta_n}^{2i\Delta_n} (b_s-b_{t_{(n,j)}})\,ds\right|^p\right] + K_p 
  \mathbb{E}\left[\left|\int_{(2i-2)\Delta_n}^{(2i-1)\Delta_n} (b_s-b_{t_{(n,j)}})\,ds\right|^p\right]
  \\
  &\quad+ K_p\mathbb{E}\left[\left| \int_{(2i-1)\Delta_n}^{2i\Delta_n} (\sigma_{s-}-\sigma_{t_{(n,j)}})\,dZ_s\right|^p\right] 
  +
  K_p\mathbb{E}\left[\left| \int_{(2i-2)\Delta_n}^{(2i-1)\Delta_n} (\sigma_{s-}-\sigma_{t_{(n,j)}})\,dZ_s\right|^p\right].
 \end{align*}
When $p\geq1$, it holds by H\"{o}lder's inequality and 
Assumption \ref{assumption-2} that 
\begin{align*}
    \mathbb{E}\left[\left|\int_{(2i-1)\Delta_n}^{{2i}\Delta_n} (b_s-b_{t_{(n,j)}})\,ds\right|^p\right]
    &\leq 
    \Delta_n^{p-1} \int_{(2i-1)\Delta_n}^{{2i}\Delta_n} \mathbb{E}\left[\left|b_s-b_{t_{(n,j)}}\right|^p\right]\,ds
    \leq  \Delta_n^{p(\tilde{\kappa}+1)}. 
\end{align*}
When $p\in(0,1)$, it holds by H\"{o}lder's inequality for $q>1$ that
\begin{align*}
     \mathbb{E}\left[\left|\int_{(2i-1)\Delta_n}^{2i\Delta_n} (b_s-b_{t_{(n,j)}})\,ds\right|^p\right] 
     &\leq  \left(\mathbb{E}\left[\left|\int_{(2i-1)\Delta_n}^{2i\Delta_n} (b_s-b_{t_{(n,j)}})\,ds\right|^q\right]\right)^{p/q} 
     \leq \Delta_n^{p(\tilde{\kappa}+1)}. 
\end{align*}
Similarly, we have that for $p>0$, 
\begin{align*}
     \mathbb{E}\left[\left|\int_{(2i-2)\Delta_n}^{(2i-1)\Delta_n} (b_s-b_{t_{(n,j)}})\,ds\right|^p\right] \leq \Delta_n^{p(\tilde{\kappa}+1)} .
\end{align*}

By Lemma \ref{le:Mom-esti} and Assumption \ref{assumption-2}, we obtain that for any $p\in(0,\beta)$, 
\begin{align*}
     \mathbb{E}\left[\left| \int_{(2i-1)\Delta_n}^{2i\Delta_n} (\sigma_{s-}-\sigma_{t_{(n,j)}})\,dZ_s\right|^p\right] 
     &\leq K_{\beta,p}  \left(\mathbb{E}\left[ \int_{(2i-1)\Delta_n}^{2i\Delta_n} \left|\sigma_{s-}-\sigma_{t_{(n,j)}}\right|^\beta\,ds\right]\right)^{p/\beta}
     \\
     &\leq K_{\beta,p} \Delta_n^{p(\kappa+1/\beta)}.
\end{align*}
Similarly, we obtain
\begin{align*}
     \mathbb{E}\left[\left| \int_{(2i-2)\Delta_n}^{(2i-1)\Delta_n} (\sigma_{s-}-\sigma_{t_{(n,j)}})\,dZ_s\right|^p\right] 
     &\leq K_{\beta,p} \Delta_n^{p(\kappa+1/\beta)}.
\end{align*}
By combing the above estimates, we obtain
\begin{align*}
 \mathbb{E}[|\Delta_{2i}^nX-\Delta_{2i-1}^nX - \sigma_{t_{(n,j)}}(\Delta_{2i}^nZ-\Delta_{2i-1}^nZ)|^p]
 \leq K_{p,\beta} \Delta_n^{p((\tilde{\kappa}+1)\wedge(\kappa+1/\beta))},
\end{align*}
which implies that
\begin{align}
\label{eq:key-esti-diff}
    |\Delta_{2i}^nX-\Delta_{2i-1}^nX - \sigma_{t_{(n,j)}}(\Delta_{2i}^nZ-\Delta_{2i-1}^nZ)| = O_\mathbb{P}(\Delta_n^{(\tilde{\kappa}+1)\wedge(\kappa+1/\beta)}).
\end{align}

For the $j$-th block, define
\begin{align*}
    \tilde{\Theta}_{n,j}(p,\beta)\equiv\dfrac{1}{\Delta_n^{p/\beta}}\sum_{i\in\tilde{\mathcal{I}}_{n,j}}|\Delta_{2i}^nZ-\Delta_{2i-1}^nZ|^p,\,\,p>0.
\end{align*}
By 
$\Delta_{2i}^nZ \overset{\mathcal{L}}=\Delta_n^{1/\beta}Z_1,\Delta_{2i-1}^nZ \overset{\mathcal{L}}=\Delta_n^{1/\beta}Z_1$ and 
the additivity property of independent stable variables, 
we obtain
\begin{align}
\label{eq:same-law-tilde-Snj}
\tilde{\Theta}_{n,j}(p,\beta)\overset{\mathcal{L}}=\sum_{i\in\tilde{\mathcal{I}}_{n,j}}|\tilde{Z_i}|^p,
\end{align}
where $(\tilde{Z}_i)_{i\in\tilde{\mathcal{I}}_{n,j}}$ are i.i.d.\ $\beta$-stable 
random variables with 
$\mathbb{E}[e^{iu\tilde{Z}_i}]\equiv e^{-|u|^{\beta}}.$

Note that we have the following 
decomposition:
\begin{align}
\label{eq:decomposition-diff}
&\frac{k}{2} \tilde{\sigma}_{n,j}(p) - \sigma_{n,t}^{p} \tilde{\Theta}_{n,j}(p,\beta) \nonumber\\
&= \frac{k}{2}\tilde{\sigma}_{n,j}(p) - \Delta_n^{p/\beta}\sigma_{t_{(n,j)}}^p \tilde{\Theta}_{n,j}(p,\beta)+(\Delta_n^{p/\beta}\sigma_{t_{(n,j)}}^p - \sigma_{n,t}^{p}) \tilde{\Theta}_{n,j}(p,\beta).
\end{align}
Since $\tilde{\Theta}_{n,j}(p,\beta)=O_{\mathbb{P}}(1)$, it holds by Assumption \ref{assumption-2} that for
$p>0$,
\begin{align}
\label{eq:second-term-diff}
   |(\Delta_n^{p/\beta}\sigma_{t_{(n,j)}}^p - \sigma_{n,t}^{p}) \tilde{\Theta}_{n,j}(p,\beta)|
&=O_{\mathbb{P}}(\Delta_n^{p/\beta+\kappa(1\wedge p)}).
\end{align}

For $p\in (0,1]$,  it holds by $|a^p - b^p| \le |a - b|^p,\,\forall a,b \ge 0$ and 
\eqref{eq:key-esti-diff} that
\begin{align*}     
    ||\Delta_{2i}^nX-\Delta_{2i-1}^nX|^p-|\sigma_{t_{(n,j)}}(\Delta_{2i}^nZ-\Delta_{2i-1}^nZ)|^p|  = O_{\mathbb{P}}(\Delta_n^{p((1+\tilde{\kappa})\wedge(\kappa+1/\beta))}).
\end{align*}
And so, the triangle inequality implies that
\begin{align}
\label{eq:first-term-smaller-1-diff}
    \frac{k}{2}\tilde{\sigma}_{n,j}(p) - \Delta_n^{p/\beta}\sigma_{t_{(n,j)}}^p \tilde{\Theta}_{n,j}(p,\beta) = O_{\mathbb{P}}(\Delta_n^{p((1+\tilde{\kappa})\wedge(\kappa+1/\beta))}).
\end{align}
Combing \eqref{eq:decomposition-diff}, \eqref{eq:second-term-diff} and \eqref{eq:first-term-smaller-1-diff}, we obtain
\begin{align*}
   \frac{k}{2} \tilde{\sigma}_{n,j}(p) - \sigma_{n,t}^{p} \tilde{\Theta}_{n,j}(p,\beta)  = O_{\mathbb{P}}(\Delta_n^{p((1+\tilde{\kappa})\wedge(\kappa+1/\beta))}).
\end{align*}
Since $\sigma_{n,t}^{p}=O_\mathbb{P}(\Delta_n^{p/\beta})$, we have
\begin{align}
\label{eq:case-smaller-1-diff}
    \dfrac{\tilde{\sigma}_{n,j}(p)}{\sigma_{n,t}^{p}} - \dfrac{2}{k} \tilde{\Theta}_{n,j}(p,\beta)  = O_{\mathbb{P}}(\Delta_n^{p(\kappa\wedge(1+\tilde{\kappa}-1/\beta))}).
\end{align}

For $p>1$, by $|a^p - b^p| \le p\,(a^{p-1} + b^{p-1})|a - b|,\,\forall a,b \ge 0$, $|\Delta_{2i}^nX-\Delta_{2i-1}^nX|=O_{\mathbb{P}}(\Delta_n^{1/\beta})$, $|\Delta_{2i}^nZ-\Delta_{2i-1}^nZ|=O_{\mathbb{P}}(\Delta_n^{1/\beta})$ 
and \eqref{eq:key-esti-diff}, we obtain
\begin{align*}
    ||\Delta_{2i}^nX-\Delta_{2i-1}^nX|^p-|\sigma_{t_{(n,j)}}(\Delta_{2i}^nZ-\Delta_{2i-1}^nZ)|^p| = O_{\mathbb{P}}(\Delta_n^{(1+\tilde{\kappa})\wedge(\kappa+1/\beta)+(p-1)/\beta}).
\end{align*}
And so, the triangle inequality implies that
\begin{align}
\label{eq:first-term-bigger-1-diff}
   \frac{k}{2}\hat{\sigma}_{n,j}(p) - \Delta_n^{p/\beta}\sigma_{t_{(n,j)}}^p \Theta_{n,j}(p,\beta) = O_{\mathbb{P}}(\Delta_n^{(1+\tilde{\kappa})\wedge(\kappa+1/\beta)+(p-1)/\beta}).
\end{align}
Combing \eqref{eq:decomposition-diff}, \eqref{eq:second-term-diff} and \eqref{eq:first-term-bigger-1-diff}, we obtain
\begin{align*}
   \frac{k}{2} \tilde{\sigma}_{n,j}(p) - \sigma_{n,t}^{p} \tilde{\Theta}_{n,j}(p,\beta)  = O_{\mathbb{P}}(\Delta_n^{(1+\tilde{\kappa})\wedge(\kappa+1/\beta)+(p-1)/\beta}).
\end{align*}
Since $\sigma_{n,t}^{p}=O_\mathbb{P}(\Delta_n^{p/\beta})$, we have
\begin{align}
\label{eq:case-bigger-1-diff}
     \dfrac{\tilde{\sigma}_{n,j}(p)}{\sigma_{n,t}^{p}} - \dfrac{2}{k} \tilde{\Theta}_{n,j}(p,\beta) = O_{\mathbb{P}}(\Delta_n^{\kappa\wedge(1+\tilde{\kappa}-1/\beta)}).
\end{align}
The proof is completed by combing \eqref{eq:same-law-tilde-Snj}, \eqref{eq:case-smaller-1-diff} and \eqref{eq:case-bigger-1-diff}.
\end{proof}

%%%%%%%%%%%%%%%%%%%%%%%%%%%%%%%%%%%%%%%%%%
%%%%%%%%% Proof of Lemma 1 %%%%%%%%%%%%%%%
%%%%%%%%%%%%%%%%%%%%%%%%%%%%%%%%%%%%%%%%%%

\begin{proof}[Proof of Lemma \ref{le:quantile-continuity}]
Let $\beta_1,\beta_2\in(0,2)$ be the activity index related to stable variables $Z^{(\beta_1)}$ and $Z^{(\beta_2)}$ satisfying
for any $u\in\mathbb{R}$ that 
$
\mathbb{E}[e^{iuZ^{(\beta_1)}}] = e^{-|u|^{\beta_1}/2}
\,\,\text{and}\,\,\mathbb{E}[e^{iuZ^{(\beta_2)}}] = e^{-|u|^{\beta_2}/2},
$
respectively. 
As $\beta_1\rightarrow\beta_2$ , it is clear that 
$\mathbb{E}[e^{iuZ^{(\beta_1)}}]\rightarrow\mathbb{E}[e^{iuZ^{(\beta_2)}}]$
for each $u\in\mathbb{R}$. Then,
by the L\'{e}vy continuity theorem, we obtain 
$
    Z^{(\beta_1)} \xrightarrow{\mathcal{L}}Z^{(\beta_2)},\,\,\text{as}\,\,\beta_1\rightarrow\beta_2.
$

For any $k\geq1$, let $(Z_i^{(\beta_1)})_{1\leq i\leq k}$ and $(Z_j^{(\beta_2)})_{1\leq j\leq k}$ be the sequence of independent copy of  $Z^{(\beta_1)}$  and $Z^{(\beta_2)}$, respectively. Then we have
\begin{align*}
    (Z_1^{(\beta_1)},\ldots,Z_k^{(\beta_1)}) \xrightarrow{\mathcal{L}}(Z_1^{(\beta_2)},\ldots,Z_k^{(\beta_2)}),\,\,\text{as}\,\,\beta_1\rightarrow\beta_2.
\end{align*}

For any $p>0$, define
\begin{align*}
    \bar{S}_{k}^{(\beta_1)}(p) \equiv \dfrac{1}{k}\sum_{i=1}^{k} |Z_i^{(\beta_1)}|^p\,\, \text{and}\,\,
     \bar{S}_{k}^{(\beta_2)}(p) \equiv \dfrac{1}{k}\sum_{j=1}^{k} |Z_j^{(\beta_2)}|^p.
\end{align*}
Since the map $T:\mathbb R^k\to\mathbb R$ with
 $
   T(x_1,\dots,x_k)=\sum_{i=1}^k |x_i|^p  
 $
 is continuous for any $p>0$, 
 by the continuous mapping theorem, we obtain
\begin{align*}
     \bar{S}_{k}^{(\beta_1)}(p)\xrightarrow{\mathcal{L}} \bar{S}_{k}^{(\beta_2)}(p),\,\,\text{as}\,\,\beta_1\rightarrow\beta_2.
\end{align*}

Let $F^{(\beta_1)}:[0,\infty)\rightarrow[0,1]$ and $F^{(\beta_2)}:[0,\infty)\rightarrow[0,1]$ be the  cumulative distribution functions of $\bar{S}_{k}^{(\beta_1)}(p) $  and $\bar{S}_{k}^{(\beta_2)}(p) $, respectively. Then, for every $x\in[0,\infty)$,
\begin{align}
\label{eq:cdf-convergence}
    F^{(\beta_1)}(x)\rightarrow F^{(\beta_2)} (x) ,\,\,\text{as}\,\,\beta_1\rightarrow\beta_2.
\end{align}

We now claim that $F^{(\beta_i)},i=1,2$ are continuous and strictly increasing on $[0,\infty)$. Note that, for any $i=1,2$ and $j=1,\ldots,k$,
$Z^{(\beta_i)}_j$ has continuous distribution and  infinitely 
differentiable density on $\mathbb{R}$
(see Theorem 1.1 in  \cite{nolan2020univariate}). Furthermore, the support of density function is $\mathbb{R}$ (see Lemma 1.1 in \cite{nolan2020univariate}).

Let 
$f_{Z_j^{(\beta_i)}}$ be the probability density function of 
$Z^{(\beta_i)}_j$. 
We have for
$x\in[0,\infty)$
\begin{align*}
\mathbb{P}\left[\left|Z^{(\beta_i)}_j\right|^p\le x\right] = \mathbb{P}\left[\left|Z^{(\beta_i)}_j\right|\le x^{1/p}\right] = \int_{-x^{1/p}}^{x^{1/p}} f_{Z_j^{(\beta_i)}}(y)\,dy.
\end{align*}
Let 
$f_{|Z_j^{(\beta_i)}|^p}$ be the probability density function of 
$|Z^{(\beta_i)}_j|^p$. We have for 
$x\in[0,\infty)$,
\begin{align*}
    f_{|Z^{(\beta_i)}_j|^p} (x)
    &= \frac{1}{p}x^{1/p-1}f_{Z_j^{(\beta_i)}}(x^{1/p}) + \frac{1}{p}x^{1/p-1} f_{Z_j^{(\beta_i)}}(-x^{1/p})
    =\frac{2}{p}x^{1/p-1}f_{Z_j^{(\beta_i)}}(x^{1/p}),
\end{align*}
where the second inequality follows from the symmetry of $f_{Z_j^{(\beta_i)}}$. It is clear  that $f_{|Z^{(\beta_i)}_j|^p}$  is strictly positive and infinitely differentiable on $(0,\infty)$  ae well as $f_{|Z^{(\beta_i)}_j|^p}\in L^{1}([0,\infty))$. 
Let $f_{\bar{S}^{(\beta_i)}_k(p)}$ be the probability density function of 
$\bar{S}^{(\beta_i)}_k(p),\,i=1,2$. Then we have for $x\in(0,\infty)$,
$$
  f_{\bar{S}^{(\beta_i)}_k(p)} (x) = f_{|Z^{(\beta_i)}_j|^p}^{*k}(x),
$$
where $*k$ denotes convolution for $k\geq2$ times. Moreover, $f_{\bar{S}^{(\beta_i)}_k(p)}$ is strictly positive and infinitely differentiable  on $(0,\infty)$ (see Propositions 8.10 and 8.11 in \cite{folland1999real}). Therefore, 
\begin{align*}
    F^{(\beta_i)}(x) = \int_0^{x} f_{\bar{S}^{(\beta_i)}_k(p)} (y) \,dy 
\end{align*}
is continuous and strictly increasing on $(0,\infty)$. On the other hand, it is clear that $F^{(\beta_i)}(x) =0$ for $x<0$. Since there is no atom at the
point $x=0$, we have $F^{(\beta_i)}(0) =0$. In addition, 
\begin{align*}
    F^{(\beta_i)}(0^{+}) = \lim_{x\downarrow 0}F^{(\beta_i)}(x) = \lim_{x\downarrow 0}\mathbb{P}[0\leq \bar{S}^{(\beta_i)}_k(p)\leq x] =0 = F^{(\beta_i)}(0),
\end{align*}
and so, $F^{(\beta_i)}$ is right continuous at $x=0$. Combing $F^{(\beta_i)}(x) =0$ for all $x\leq 0$, we have $F^{(\beta_i)}$
is continuous and strictly increasing on $[0,\infty)$.

For any $\alpha\in(0,1)$,
let $\bar{q}_{\alpha}(\beta_1)$ and $\bar{q}_{\alpha}(\beta_2)$ be the $\alpha$-quantile of 
$\bar{S}_{k}^{(\beta_1)}(p)$ and $\bar{S}_{k}^{(\beta_2)}(p)$, respectively.
{Since $F^{(\beta_2)}$ is continuous and strictly increasing}, 
for any $\epsilon>0$, we have
\[
F^{(\beta_2)}(\bar{q}_\alpha(\beta_2)-\epsilon)<\alpha<
F^{(\beta_2)}(\bar{q}_\alpha(\beta_2)+\epsilon).
\]
Define
\[
\epsilon_1\equiv\alpha-F^{(\beta_2)}(\bar{q}_\alpha(\beta_2)-\epsilon)>0,
\quad 
\epsilon_2\equiv F^{(\beta_2)}(\bar{q}_\alpha(\beta_2)+\epsilon)-\alpha>0,\quad\text{and}\quad \epsilon_0\equiv\min(\epsilon_1,\epsilon_2).
\] 
Since $F^{(\beta_1)}(x)\to F^{(\beta_2)}(x)$ as $\beta_1\to\beta_2$ for each fixed $x$,
there exists $\delta>0$ such that whenever $|\beta_1-\beta_2|<\delta$,
\[
|F^{(\beta_1)}(\bar{q}_\alpha(\beta_2)\pm\epsilon)-F^{(\beta_2)}(\bar{q}_\alpha(\beta_2)\pm\epsilon)|
<\epsilon_0.
\]

 For the left point \(\bar{q}_\alpha(\beta_2)-\epsilon\), since \(\epsilon_0\le\epsilon_1\), we have
\begin{align*}
F^{(\beta_1)}(\bar{q}_\alpha(\beta_2)-\epsilon)
&= F^{(\beta_2)}(\bar{q}_\alpha(\beta_2)-\epsilon)
+ \big(F^{(\beta_1)}(\bar{q}_\alpha(\beta_2)-\epsilon)-F^{(\beta_2)}(\bar{q}_\alpha(\beta_2)-\epsilon)\big)\\
&< F^{(\beta_2)}(\bar{q}_\alpha(\beta_2)-\epsilon) + \epsilon_0
= \alpha - \epsilon_1 + \epsilon_0
\leq \alpha.
\end{align*}
For the right point $\bar{q}_\alpha(\beta_2)+\epsilon$, since \(\epsilon_0\le\epsilon_1\), we have
\begin{align*}
F^{(\beta_1)}(\bar{q}_\alpha(\beta_2)+\epsilon)
&= F^{(\beta_2)}(\bar{q}_\alpha(\beta_2)+\epsilon)
+ \big(F^{(\beta_1)}(\bar{q}_\alpha(\beta_2)+\epsilon)-F^{(\beta_2)}(\bar{q}_\alpha(\beta_2)+\epsilon)\big)\\
&> F^{(\beta_2)}(\bar{q}_\alpha(\beta_2)+\epsilon) - \epsilon_0
= \alpha + \epsilon_2 - \epsilon_0
\geq \alpha.
\end{align*}
Combining the above estimates, we obtain
\[
F^{(\beta_1)}(\bar{q}_\alpha(\beta_2)-\epsilon) < \alpha < F^{(\beta_1)}(\bar{q}_\alpha(\beta_2)+\epsilon).
\]
Since $F^{(\beta_1)}$ is strictly increasing, it follows that
\[
\bar{q}_\alpha(\beta_2)-\epsilon < \bar{q}_\alpha(\beta_1) < \bar{q}_\alpha(\beta_2)+\epsilon,
\]
which implies that $\bar{q}_\alpha(\beta_1)\to \bar{q}_\alpha(\beta_2)$
as $\beta_1\to\beta_2$.

For the continuity of $\tilde{q}_{\alpha}(\beta)$, since $(\tilde{Z}_i)_{i\in\tilde{\mathcal{I}}_{n,j}}$ is also a sequence of i.i.d. symmetric $\beta$-stable variables,  
the rest of proof is same as in the case of 
$\bar{q}_{\alpha}(\beta)$, and we omit it.
\end{proof}

\subsection{Proofs in {\rm Section \ref{sect:large-k}}}

%%%%%%%%%%%%%%%%%%%%%%%%%%%%%%%%%%%%%%%%%%
%%%%%%%%% Proof of Theorem 3 %%%%%%%%%%%%%
%%%%%%%%%%%%%%%%%%%%%%%%%%%%%%%%%%%%%%%%%%

\begin{proof}[Proof of Theorem \ref{th:large-k-1}]
Note that
we have the following decomposition:
\begin{align}
\label{eq:large-k-decomp}
    &k_{n}^{1/2}(\hat{\sigma}_{n,j}(p,\beta) - \sigma_{n,t}^p) - \sigma_{n,t}^p\Xi_{n,j}(p,\beta) \nonumber\\
    &=  k_{n}^{1/2}(\Delta_n^{p/\beta}\sigma_{t_{(n,j)}}^p - \Delta_n^{p/\beta}\sigma_{t}^p) 
    +(\Delta_n^{p/\beta}\sigma_{t_{(n,j)}}^p - \Delta_n^{p/\beta}\sigma_{t}^p) \Xi_{n,j}(p,\beta)
    + R_{n,j},
\end{align}
where 
\begin{align*}
   R_{n,j} &\equiv k_{n}^{1/2}(\hat{\sigma}_{n,j}(p,\beta)-\Delta_n^{p/\beta}\sigma_{t_{(n,j)}}^p ) - \Delta_n^{p/\beta}\sigma_{t_{(n,j)}}^p\Xi_{n,j}(p,\beta). 
  \end{align*}

 For the first term in \eqref{eq:large-k-decomp}, we obtain that for $0<p<\beta/2<1$,
\begin{align}
\label{eq:esti-first}
    \mathbb{E}[|k_{n}^{1/2}(\Delta_n^{p/\beta}\sigma_{t_{(n,j)}}^p - \Delta_n^{p/\beta}\sigma_{t}^p)|] 
    &\leq Kk_{n}^{1/2}\Delta_n^{p/\beta}\mathbb{E}[|\sigma_{t_{(n,j)}} - \sigma_t|^p]
    \nonumber\\
    &\leq Kk_{n}^{1/2}\Delta_n^{p/\beta}(k_{n}\Delta_n)^{p\kappa}
   \nonumber \\
    &\leq K \Delta_n^{p\left(\kappa +\frac{1}{\beta}-(\kappa+\frac{1}{2p})\gamma\right)},
\end{align}
where the first inequality follows by 
$|a^p-b^p|\leq |a-b|^p,\forall a,b\geq0$,
the second inequality follows by the condition 
(ii) in Assumption \ref{assumption}, and the last inequality follows from $k_{n} \asymp\Delta_n^{-\gamma}$.

For the second term in \eqref{eq:large-k-decomp}, since $\Xi_{n,j}(p,\beta)=O_{\mathbb{P}}(1)$, we have
\begin{align}
\label{eq:esti-second}
   \mathbb{E}[|(\Delta_n^{p/\beta}\sigma_{t_{(n,j)}}^p - \Delta_n^{p/\beta}\sigma_{t}^p)\Xi_{n,j}(p,\beta)| ] 
   &\leq 
   K\Delta_n^{p/\beta}(k_{n}\Delta_n)^{p\kappa} 
\leq 
K \Delta_n^{p\left(\kappa +\frac{1}{\beta}-\kappa \gamma\right)}.
\end{align}

For $R_{n,j}$, it holds by (\ref{eq:large-k-estimator}) and
the definition of $\Xi_{n,j}(p,\beta)$ that
 \begin{align*}  
   R_{n,j}
   &= k_{n}^{1/2}\left(\frac{1}{c_{\beta}(p)k_{n}}\sum_{i \in \mathcal{I}_{n,j}} |\Delta_i^n X|^p-\Delta_n^{p/\beta}\sigma_{t_{(n,j)}}^p \right) 
 - 
   \frac{\Delta_n^{p/\beta}\sigma_{t_{(n,j)}}^p}{k_{n}^{1/2}} \sum_{i\in\mathcal{I}_{n,j}}\left(\dfrac{|\Delta_i^n Z|^p}{c_{\beta}(p)\Delta_n^{p/\beta}}-1\right)
   \\
    &= k_{n}^{-1/2}c_{\beta}(p)^{-1} \sum_{i\in \mathcal{I}_{n,j}} (|\Delta_i^n X|^p-|\sigma_{t_{(n,j)}}\Delta_i^n Z|^p).
\end{align*}
And so, the triangle inequality implies that
\begin{align*}
    \mathbb{E}[|R_{n,j}|] &\leq K_{p,\beta}k_{n}^{-1/2} \sum_{i\in\mathcal{I}_{n,j}} \mathbb{E}[|\Delta_i^n X-\sigma_{t_{(n,j)}}\Delta_i^nZ|^p].
\end{align*}

From \eqref{eq:Price-process}, we have
\begin{align}
\label{eq:core-inequality-stable}
    \mathbb{E}[|\Delta_i^n X-\sigma_{t_{(n,j)}}\Delta_i^n Z|^p]
   & \leq K_p
    \mathbb{E}\left[\left|\int_{(i-1)\Delta_n}^{i\Delta_n} b_s \,ds\right|^p\right]
 +K_p
    \mathbb{E}\left[\left|\int_{(i-1)\Delta_n}^{i\Delta_n} (\sigma_{s-}-\sigma_{t_{(n,j)}}) \,dZ_s\right|^p\right]
     \nonumber\\
    &\leq K_p\Delta_n^{p}+K_{p\beta}\left(\int_{(i-1)\Delta_n}^{i\Delta_n}\mathbb{E}[ |\sigma_{s}-\sigma_{t_{(n,j)}}|^{\beta} ]\,ds\right)^{p/\beta}
    \nonumber \\
    &\leq K_p\Delta_n^{p}+K_{p,\beta}((k_{n}\Delta_n)^{\beta \kappa}\Delta_n)^{p/\beta}
    \nonumber \\
    &\leq K_{p,\beta}\Delta_n^{p(1\wedge(\kappa+1/\beta-\kappa\gamma))},
\end{align}
where the second inequality follows from Assumption \ref{assumption} (i)
and Lemma \ref{le:Mom-esti}, the third inequality follows by Assumption \ref{assumption} (ii),
and the last inequality follows by  $k_{n} \asymp\Delta_n^{-\gamma}$.
Subsequently, we have
\begin{align}
\label{eq:R_nj-smaller-1}
    \mathbb{E}[|R_{n,j}|] 
    \leq K_{p,\beta}k_{n}^{1/2}\Delta_n^{p(1\wedge(\kappa+1/\beta-\kappa\gamma))}
    % \nonumber\\
    % & =K_{p,\beta}\Delta_n^{p(1\wedge(\kappa+1/\beta-\kappa\gamma))-\gamma/2}
    % \nonumber\\
\leq K_{p,\beta}\Delta_n^{p\left(\left(1-\frac{\gamma}{2p}\right)\wedge\left(\kappa+\frac{1}{\beta}-(\kappa+\frac{1}{2p})\gamma\right)\right)}.
\end{align}

Hence,
by combing \eqref{eq:large-k-decomp}, \eqref{eq:esti-first}, \eqref{eq:esti-second} and \eqref{eq:R_nj-smaller-1}, we obtain
\begin{align*}
  \mathbb{E} \left[ \left|k_{n}^{1/2}(\hat{\sigma}_{n,j}(p,\beta) - \sigma_{n,t}^p) - \sigma_{n,t}^p\Xi_{n,j}(p,\beta)\right|\right] \leq K_{p,\beta} \Delta_n^{p\left(\left(1-\frac{\gamma}{2p}\right)\wedge\left(\kappa+\frac{1}{\beta}-(\kappa-\frac{1}{2p})\gamma\right)\right)}.
\end{align*}
Since $\sigma_{n,t}^p = O_{\mathbb{P}}(\Delta_n^{p/\beta})$, $\beta\in(1,2)$ and $0<\gamma <2p(1-\frac{1}{\beta})\wedge\frac{\kappa}{\kappa+\frac{1}{2p}}$,
we obtain
 \begin{align}
 \label{eq:point-coupling}
    \dfrac{k_{n}^{1/2}(\hat{\sigma}_{n,j}(p,\beta) - \sigma_{n,t}^p)}{\sigma_{n,t}^p} - \Xi_{n,j}(p,\beta) 
    % &= O_{\mathbb{P}}\left(\Delta_n^{p\left(\left(1-\frac{1}{\beta}-\frac{\gamma}{2p}\right)\wedge\left(\kappa-(\kappa+\frac{1}{2p})\gamma\right)\right)}\right) 
    % \nonumber\\
    &= o_{\mathbb{P}}\left(\Delta_n^{\varepsilon}\right),
\end{align}
where $0<\varepsilon<p((1-\frac{1}{\beta}-\frac{\gamma}{2p})\wedge(\kappa-(\kappa+\frac{1}{2p})\gamma))$.

For $f(x) = \log (x)$, we have $f'(x)x=1$.
By the mean-value theorem, it holds for some $\xi_{n,j}(p,\beta)$
between $\hat{\sigma}_{n,j}(p,\beta)$ and $\sigma_{n,t}^p$ that 
\begin{align*}
    k_n^{1/2} (\log(\hat{\sigma}_{n,j}(p,\beta))-\log (\sigma_{n,t}^p)) 
    &=  \dfrac{k_{n}^{1/2}(\hat{\sigma}_{n,j}(p,\beta) - \sigma_{n,t}^p)}{\sigma_{n,t}^p} -\dfrac{k_n^{1/2}|\hat{\sigma}_{n,j}(p,\beta) - \sigma_{n,t}^p|^2}{2\xi_{n,j}(p,\beta)^2}
    \\
    & = \Xi_{n,j}(p,\beta) + o_{\mathbb{P}}(\Delta_n^{\varepsilon})-\dfrac{k_n^{1/2}|\hat{\sigma}_{n,j}(p,\beta) - \sigma_{n,t}^p|^2}{2\xi_{n,j}(p,\beta)^2},
\end{align*}
where the second equality follows from \eqref{eq:point-coupling}.
Since $\sigma_{n,t}^p = O_{\mathbb{P}}(\Delta_n^{p/\beta})$, $\Xi_{n,j}(p,\beta)=O_{\mathbb{P}}(1)$ and $k_{n}\asymp\Delta_n^{-\gamma}$, 
we obtain  from \eqref{eq:point-coupling} that $k_{n}^{1/2}|\hat{\sigma}_{n,j}(p,\beta) - \sigma_{n,t}^p|^2 =  O_{\mathbb{P}}(\Delta_n^{2p/\beta+\gamma/2})$. 
On the other hand, Since $\xi_{n,j}(p,\beta)$ lies between $\hat{\sigma}_{n,j}(p,\beta)$ and $\sigma_{n,t}^p$, and by \eqref{eq:point-coupling} we have
$\hat{\sigma}_{n,j}(p,\beta)=\sigma_{n,t}^p\bigl(1+O_{\mathbb P}(\Delta_n^{\gamma/2})\bigr)$, it follows that
$\xi_{n,j}(p,\beta)=O_{\mathbb{P}}(\Delta_n^{p/\beta})$. Subsequently, 
we obtain 
\begin{align*}
   \dfrac{k_n^{1/2}|\hat{\sigma}_{n,j}(p,\beta) - \sigma_{n,t}^p|^2}{2\xi_{n,j}(p,\beta)^2} = O_{\mathbb{P}}(\Delta_n^{\gamma/2}).
\end{align*}
Hence, we obtain that for some $\varepsilon>0$,
\begin{align*}
   k_n^{1/2} (\log(\hat{\sigma}_{n,j}(p,\beta))-\log (\sigma_{n,t}^p))  - \Xi_{n,j}(p,\beta) =  o_{\mathbb{P}}(\Delta_n^{\varepsilon}).
\end{align*}

For $f(x) = x^{r},r>0$, we have $f'(x) = rx^{r-1}$ and $f''(x)=r(r-1)x^{r-2}$.
Similarly, it holds by the mean-value theorem, \eqref{eq:point-coupling} 
and $\sigma_{n,t}^p = O_{\mathbb{P}}(\Delta_n^{p/\beta})$ that
\begin{align*}
    &k_n^{1/2} (f(\hat{\sigma}_{n,j}(p,\beta))-f (\sigma_{n,t}^p)) 
    \\
    &=  f'(\sigma_{n,t}^p){k_{n}^{1/2}(\hat{\sigma}_{n,j}(p,\beta) - \sigma_{n,t}^p)} +\dfrac{f''(\xi_{n,j}(p,\beta))k_n^{1/2}|\hat{\sigma}_{n,j}(p,\beta) - \sigma_{n,t}^p|^2}{2}
    \\
    & = f'(\sigma_{n,t}^p)(\sigma_{n,t}^p\Xi_{n,j}(p,\beta)+o_{\mathbb{P}}(\Delta_n^{p/\beta+\varepsilon})) +\dfrac{f''(\xi_{n,j}(p,\beta))k_n^{1/2}|\hat{\sigma}_{n,j}(p,\beta) - \sigma_{n,t}^p|^2}{2}.
\end{align*}
% Note that 
% $f'(\sigma_{n,t}^p) = O_{\mathbb{P}}(\Delta_n^{(r-1)p/\beta})$ and 
Since $f''(\xi_{n,j}(p,\beta)) = O_{\mathbb{P}}(\Delta_n^{(r-2)p/\beta})$
and  $k_{n}^{1/2}|\hat{\sigma}_{n,j}(p,\beta) - \sigma_{n,t}^p|^2 =  O_{\mathbb{P}}(\Delta_n^{2p/\beta+\gamma/2})$, we obtain 
\begin{align*}
    \dfrac{f''(\xi_{n,j}(p,\beta))k_n^{1/2}|\hat{\sigma}_{n,j}(p,\beta) - \sigma_{n,t}^p|^2}{2} = O_{\mathbb{P}}(\Delta_n^{rp/\beta+\gamma/2}).
\end{align*}
Furthermore, since $f'(\sigma_{n,t}^p) = O_{\mathbb{P}}(\Delta_n^{(r-1)p/\beta})$, we obtain that for some $\varepsilon>0$,
\begin{align}
\label{eq:infeasible-transform-f}
    k_n^{1/2} (f(\hat{\sigma}_{n,j}(p,\beta))-f (\sigma_{n,t}^p)) -
    f'(\sigma_{n,t}^p)\sigma_{n,t}^p\Xi_{n,j}(p,\beta) = o_{\mathbb{P}} (\Delta_n^{rp/\beta+\varepsilon}).
\end{align}
Note that we have 
\begin{align*}
    \dfrac{  k_n^{1/2} (f(\hat{\sigma}_{n,j}(p,\beta))-f (\sigma_{n,t}^p)) }{ f'(\hat{\sigma}_{n,j}(p,\beta))\hat{\sigma}_{n,j}(p,\beta)} -\Xi_{n,j}(p,\beta) = A_{1,n} + A_{2,n},
\end{align*}
where
\begin{align*}
 A_{1,n} &= \left(\dfrac{f'({\sigma}_{n,t}^p){\sigma}_{n,t}^p}{f'(\hat{\sigma}_{n,j}(p,\beta))\hat{\sigma}_{n,j}(p,\beta)}-1\right)\Xi_{n,j}(p,\beta)=\left(\left(\dfrac{{\sigma}_{n,t}^{p}}{\hat{\sigma}_{n,j}(p,\beta)}\right)^{r}-1\right)\Xi_{n,j}(p,\beta) ,
 \\
  A_{2,n} &=  \dfrac{ k_n^{1/2} (f(\hat{\sigma}_{n,j}(p,\beta))-f (\sigma_{n,t}^p))- f'(\sigma_{n,t}^p)\sigma_{n,t}^p\Xi_{n,j}(p,\beta)}{f'(\hat{\sigma}_{n,j}(p,\beta))\hat{\sigma}_{n,j}(p,\beta)}.
\end{align*}
By $\hat{\sigma}_{n,j}(p,\beta)=\sigma_{n,t}^p\bigl(1+O_{\mathbb P}(\Delta_n^{\gamma/2})\bigr)$ and $\Xi_{n,j}(p,\beta)=O_{\mathbb{P}}(1)$, we have $A_{1,n} = o_{\mathbb{P}}(\Delta_n^{\varepsilon})$ for some $\epsilon>0$. 
For $A_{2,n}$, it holds by $\sigma_{n,t}^p=O_{\mathbb{P}}(\Delta_n^{p/\beta})$ that $f'(\hat{\sigma}_{n,j}(p,\beta))\hat{\sigma}_{n,j}(p,\beta) = O_{\mathbb{P}}(\Delta_n^{rp/\beta})$. Hence,  by combing \eqref{eq:infeasible-transform-f}, we also have $A_{2,n}= o_{\mathbb{P}}(\Delta_n^{\varepsilon})$ for some $\epsilon>0$. The proof is completed.
\end{proof}

%%%%%%%%%%%%%%%%%%%%%%%%%%%%%%%%%%%%%%%%%%
%%%%%%%%% Proof of Theorem 4 %%%%%%%%%%%%%
%%%%%%%%%%%%%%%%%%%%%%%%%%%%%%%%%%%%%%%%%%

\begin{proof}[Proof of Theorem \ref{th:large-k-1-stable}]
Note that
we have the following  decomposition:
\begin{align}
\label{eq:large-k-decomp-stable}
    &k_{n}^{1-p/\beta}(\hat{\sigma}_{n,j}(p,\beta) - \sigma_{n,t}^p) - \sigma_{n,t}^p{S}_{n,j}(p,\beta) \nonumber\\
    &=  k_{n}^{1-p/\beta}(\Delta_n^{p/\beta}\sigma_{t_{(n,j)}}^p - \Delta_n^{p/\beta}\sigma_{t}^p) 
    +(\Delta_n^{p/\beta}\sigma_{t_{(n,j)}}^p - \Delta_n^{p/\beta}\sigma_{t}^p) {S}_{n,j}(p,\beta)
    + R_{n,j}',
\end{align}
where 
\begin{align*}
    R_{n,j}' &=k_{n}^{1-p/\beta}(\hat{\sigma}_{n,j}(p,\beta)-\Delta_n^{p/\beta}\sigma_{t_{(n,j)}}^p ) - \Delta_n^{p/\beta}\sigma_{t_{(n,j)}}^p{S}_{n,j}(p,\beta). 
  \end{align*}

For $p\in(\beta/2,1)$, we have
\begin{align}
\label{eq:esti-first-stable}
    \mathbb{E}[|k_{n}^{1-p/\beta}(\Delta_n^{p/\beta}\sigma_{t_{(n,j)}}^p - \Delta_n^{p/\beta}\sigma_{t}^p)|] 
    &\leq 
    Kk_{n}^{1-p/\beta}\Delta_n^{p/\beta}(k_{n}\Delta_n)^{p\kappa}
   \nonumber \\
    &\leq K \Delta_n^{p\left(\kappa +\frac{1}{\beta}-(\kappa+\frac{1}{p}-\frac{1}{\beta})\gamma\right)}.
\end{align}
Since ${S}_{n,j}(p,\beta)=O_{\mathbb{P}}(1)$, we also have
\begin{align}
\label{eq:esti-second-stable}
   \mathbb{E}[|(\Delta_n^{p/\beta}\sigma_{t_{(n,j)}}^p - \Delta_n^{p/\beta}\sigma_{t}^p){S}_{n,j}(p,\beta)| ] &\leq  K \Delta_n^{p\left(\kappa +\frac{1}{\beta}-\kappa \gamma\right)}.
\end{align}
For ${R_{n,j}'}$, it holds by \eqref{eq:large-k-estimator} and
the definition of $S_{n,j}(p,\beta)$ that
 \begin{align*}  
   R_{n,j}'
   &= k_{n}^{1-p/\beta}\left(\frac{1}{c_{\beta}(p)k_{n}}\sum_{i \in \mathcal{I}_{n,j}} |\Delta_i^n X|^p-\Delta_n^{p/\beta}\sigma_{t_{(n,j)}}^p \right) 
  - 
   \frac{\Delta_n^{p/\beta}\sigma_{t_{(n,j)}}^p}{k_{n}^{p/\beta}} \sum_{i\in\mathcal{I}_{n,j}}\left(\dfrac{|\Delta_i^n Z|^p}{c_{\beta}(p)\Delta_n^{p/\beta}}-1\right)
    \\
    &= k_{n}^{-p/\beta}
    c_{\beta}(p)^{-1} \sum_{i\in \mathcal{I}_{n,j}} (|\Delta_i^n X|^p-|\sigma_{t_{(n,j)}}\Delta_i^n Z|^p).
\end{align*}
And so, the triangle inequality and \eqref{eq:core-inequality-stable} imply that
\begin{align}
\label{eq:R_nj-stable}
    \mathbb{E}[|R_{n,j}'|] 
    &\leq K_{p,\beta}k_{n}^{-p/\beta} \sum_{i\in\mathcal{I}_{n,j}} \mathbb{E}[|\Delta_i^n X-\sigma_{t_{(n,j)}}\Delta_i^nZ|^p]
    \nonumber\\
&\leq K_{p,\beta}\Delta_n^{p\left(\left(1-(\frac{1}{p}-\frac{1}{\beta})\gamma\right)\wedge\left(\kappa+\frac{1}{\beta}-(\kappa+\frac{1}{p}-\frac{1}{\beta})\gamma\right)\right)}.
\end{align}
Combing \eqref{eq:large-k-decomp-stable}, \eqref{eq:esti-first-stable}, \eqref{eq:esti-second-stable} and \eqref{eq:R_nj-stable}, we obtain
\begin{align*}
  \mathbb{E} \left[ \left|k_{n}^{1-1/\beta}(\hat{\sigma}_{n,j}(p,\beta) - \sigma_{n,t}^p) - \sigma_{n,t}^p{S}_{n,j}(p,\beta)\right|\right] \leq
 K_{p,\beta}\Delta_n^{p\left(\left(1-(\frac{1}{p}-\frac{1}{\beta})\gamma\right)\wedge\left(\kappa+\frac{1}{\beta}-(\kappa+\frac{1}{p}-\frac{1}{\beta})\gamma\right)\right)}.
\end{align*}
Since $\sigma_{n,t}^p = O_{\mathbb{P}}(\Delta_n^{p/\beta})$, $\beta\in(1,2)$ and $0<\gamma <\frac{1-\frac{1}{\beta}}{\frac{1}{p}-\frac{1}{\beta}}\wedge\frac{\kappa}{\kappa+\frac{1}{p}-\frac{1}{\beta}}$, we have 
 \begin{align*}
    \dfrac{k_{n}^{1-1/\beta}(\hat{\sigma}_{n,j}(p,\beta) - \sigma_{n,t}^p)}{\sigma_{n,t}^p} - {S}_{n,j}(p,\beta) 
    % &= O_{\mathbb{P}}\left(\Delta_n^{p\left(\left(1-\frac{1}{\beta}-(\frac{1}{p}-\frac{1}{\beta})\gamma\right)\wedge\left(\kappa-(\kappa+\frac{1}{p}-\frac{1}{\beta})\gamma\right)\right)}\right) 
    % \\
    &= o_{\mathbb{P}}\left(\Delta_n^{\varepsilon}\right),
\end{align*}
where $0<\varepsilon<p\left(\left(1-\frac{1}{\beta}-(\frac{1}{p}-\frac{1}{\beta})\gamma\right)\wedge\left(\kappa-(\kappa+\frac{1}{p}-\frac{1}{\beta})\gamma\right)\right)$.

For $p\in[1,\beta)$, it holds by the inequality $|a^p-b^p| \leq p (a^{p-1}+b^{p-1})|a-b|,\,\forall a,b\ge0$ that
\begin{align}
\label{eq:esti-first-stable-big-1}
    \mathbb{E}[|k_{n}^{1-p/\beta}(\Delta_n^{p/\beta}\sigma_{t_{(n,j)}}^p - \Delta_n^{p/\beta}\sigma_{t}^p)|] 
   %  &\leq Kk_{n}^{1-p/\beta}\Delta_n^{p/\beta}(k_{n}\Delta_n)^{\kappa}
   % \nonumber \\
    &\leq K \Delta_n^{\kappa +\frac{p}{\beta}-(\kappa+1-\frac{p}{\beta})\gamma},
\end{align}
and that
\begin{align}
\label{eq:esti-second-stable-big-1}
   \mathbb{E}[|(\Delta_n^{p/\beta}\sigma_{t_{(n,j)}}^p - \Delta_n^{p/\beta}\sigma_{t}^p){S}_{n,j}(p,\beta)| ] &\leq  K \Delta_n^{\kappa +\frac{p}{\beta}-\kappa \gamma}.
\end{align}
For $R_{n,j}'$, we have for triangle inequality that
\begin{align*}
    \mathbb{E}[|R_{n,j}'|] &\leq K_{p,\beta}k_{n}^{-p/\beta} \sum_{i\in\mathcal{I}_{n,j}} \mathbb{E}[||\Delta_i^n X|^p-|\sigma_{t_{(n,j)}}\Delta_i^nZ|^p|].
\end{align*}
Note that for $p\geq1$,
\begin{align*}
    &\mathbb{E}\left[\left||\Delta_i^n X|^p-|\sigma_{t_{(n,j)}}\Delta_i^n Z|^p \right|\right]
    \\
    &\leq K_p \mathbb{E}\left[(|\Delta_i^n X|^{p-1}+|\Delta_i^n Z|^{p-1})|\Delta_i^n X-\sigma_{t_{(n,j)}}\Delta_i^n Z|\right]
    \\
    &\leq K_p \left(\mathbb{E}\left[(|\Delta_i^n X|^{p-1}+|\Delta_i^n Z|^{p-1})^{\frac{p}{p-1}}\right]\right)^{\frac{p-1}{p}}(\mathbb{E}[|\Delta_i^n X-\sigma_{t_{(n,j)}}\Delta_i^n Z|^p])^{1/p}
    \\
    &\leq K_p \left(\mathbb{E}\left[|\Delta_i^n X|^{p}+|\Delta_i^n Z|^{p}\right]\right)^{\frac{p-1}{p}}
    (\mathbb{E}[|\Delta_i^n X-\sigma_{t_{(n,j)}}\Delta_i^n Z|^p])^{1/p},
\end{align*}
where the second inequality follows from H\"{o}lder's inequality 
and the third inequality follows from the inequality $(a+b)^p\leq 2^{p-1} (a^p+b^p), \,\forall a,b\geq0$. 
As in \eqref{eq:core-inequality-stable}, we have $$(\mathbb{E}[|\Delta_i^n X-\sigma_{t_{(n,j)}}\Delta_i^n Z|^p])^{1/p} \leq K_{p,\beta}\Delta_n^{1\wedge(\kappa+\frac{1}{\beta}-\kappa\gamma)}.$$
On the other hand,
since $|\Delta_i^n X|=O_{\mathbb{P}}(\Delta_n^{1/\beta})$ and $|\Delta_i^n Z|=O_{\mathbb{P}}(\Delta_n^{1/\beta})$, we have
\begin{align*}
    \mathbb{E}\left[\left||\Delta_i^n X|^p-|\sigma_{t_{(n,j)}}\Delta_i^n Z|^p \right|\right] \leq K_{p,\beta}\Delta_n^{(1+\frac{p-1}{\beta})\wedge(\kappa+\frac{p}{\beta}-\kappa\gamma)},
\end{align*}
which implies that
\begin{align}
\label{eq:R_nj-stable-big-1}
\mathbb{E}[|R_{n,j}'|] 
&\leq K_{p,\beta}k_{n}^{1-p/\beta} \Delta_n^{(1+{(p-1)}/{\beta})\wedge(\kappa+\frac{p}{\beta}-\kappa\gamma)}
\nonumber\\
&\leq K_{p,\beta} \Delta_n^{(1+\frac{p-1}{\beta}-(1-\frac{p}{\beta})\gamma)\wedge(\kappa+\frac{p}{\beta}-(\kappa
+1-\frac{p}{\beta})\gamma)}.
\end{align}
Combing \eqref{eq:large-k-decomp-stable}, \eqref{eq:esti-first-stable-big-1}, \eqref{eq:esti-second-stable-big-1} and \eqref{eq:R_nj-stable-big-1}, we obtain
\begin{align*}
  \mathbb{E} \left[ \left|k_{n}^{1-1/\beta}(\hat{\sigma}_{n,j}(p,\beta) - \sigma_{n,t}^p) - \sigma_{n,t}^p{S}_{n,j}(p,\beta)\right|\right] 
  \leq
  K_{p,\beta} \Delta_n^{(1+\frac{p-1}{\beta}-(1-\frac{p}{\beta})\gamma)\wedge(\kappa+\frac{p}{\beta}-(\kappa
+1-\frac{p}{\beta})\gamma)}.
\end{align*}
Since $\sigma_{n,t}^p = O_{\mathbb{P}}(\Delta_n^{p/\beta})$, $\beta\in(1,2)$, $p\in[1,\beta)$ and $0<\gamma <\frac{\kappa}{\kappa+1-\frac{p}{\beta}}=\frac{\kappa}{\kappa+1-\frac{p}{\beta}}\wedge\frac{1-\frac{1}{\beta}}{1-\frac{p}{\beta}}$,
we have 
 \begin{align*}
    \dfrac{k_{n}^{1-p/\beta}(\hat{\sigma}_{n,j}(p,\beta) - \sigma_{n,t}^p)}{\sigma_{n,t}^p} - {S}_{n,j}(p,\beta) 
    % &= O_{\mathbb{P}}\left(\Delta_n^{\left(1-\frac{1}{\beta}-(1-\frac{p}{\beta})\gamma\right)\wedge\left(\kappa-(\kappa+1-\frac{p}{\beta})\gamma\right)}\right) 
    % \\
    &= o_{\mathbb{P}}\left(\Delta_n^{\varepsilon}\right),
\end{align*}
where $0<\varepsilon<{\left(1-\frac{1}{\beta}-(1-\frac{p}{\beta})\gamma\right)\wedge\left(\kappa-(\kappa+1-\frac{p}{\beta})\gamma\right)}$.
The rest proof for transform $f$ is similar to the proof in Theorem \ref{th:large-k-1}, and we omit it.
\end{proof}

%%%%%%%%%%%%%%%%%%%%%%%%%%%%%%%%%%%%%%%%%%
%%%%%%%%% Proof of Lemma 2 %%%%%%%%%%%%%%%
%%%%%%%%%%%%%%%%%%%%%%%%%%%%%%%%%%%%%%%%%%

\begin{proof}[Proof of Lemma \ref{le:lim-distri-S_nj}]
    By the self-similar property of stable process, we know that
    \begin{align*}
        S_{n,j}(p,\beta)\overset{\mathcal{L}}{=} \dfrac{1}{k_{n}^{p/\beta}}\sum_{i\in\mathcal{I}_{n,j}} \left(\dfrac{|Z_i|^p}{\mathbb{E}[|Z_i|^p]}-1\right),
    \end{align*}
    where $(Z_i)_{i\in\mathcal{I}_{n,j}}$
    are i.i.d. stable  variables
   with $Z_i\sim\mathcal{S}(\beta,0,2^{-1/\beta},0)$. For any
   $i\in\mathcal{I}_{n,j}$, let us set
   $\xi_i\equiv {|Z_i|^p}/{c_{\beta}(p)}$. Then $(\xi_i)_{i\in\mathcal{I}_{n,j}}$
   are i.i.d. variables with
   $\mathbb{E}[\xi_i] =1$.
   Note that 
   \begin{align*}
\lim_{x\rightarrow\infty}x^{\beta/p}\mathbb{P}[\xi_i>x] 
    &=  \lim_{x\rightarrow\infty} x^{\beta/p} \mathbb{P}[|Z_i|>(xc_{\beta}(p))^{1/p}]
= 
\dfrac{\Gamma(\beta)\sin(\pi\beta/2)}{\pi} c_{\beta}(p)^{-\beta/p},
   \end{align*}
   where the second equality follows from
 Theorem 1.2 in  \cite{nolan2020univariate}.
% Furthermore, it is clear that
% $\lim_{x\rightarrow\infty}x^{\beta/p}\mathbb{P}[\xi_i<-x] =0$.

By taking $\eta=\beta/p$, $c^-=0$ and
$c^+={c_{\beta}(p)^{-\beta/p}\Gamma(\beta)\sin(\pi\beta/2)}/{\pi} ,$
we have
\begin{align*}
   \delta &\equiv  \dfrac{c^+-c^-}{c^++c^-} =1,
   \\
    a_n &\equiv k_{n}^{-1/\eta}\left(\dfrac{2\Gamma(\eta)\sin(\pi\eta/2)}{\pi(c^++c^-)}\right)^{1/\eta} = C_{\beta}(p) k_{n}^{-p/\beta},
\\
   b_n&\equiv k_{n} a_n \mathbb{E}[\xi] = k_{n} C_{\beta}(p) k_{n}^{-p/\beta}.
\end{align*}
% where 
% \begin{align*}
%     C_{\beta}(p) =  c_{\beta}(p)\left(\dfrac{2\Gamma(\beta/p)\sin(\pi\beta/(2p))}{\Gamma(\beta)\sin(\pi\beta/2)}\right)^{p/\beta}.
% \end{align*}
Hence, by the Generalized CLT (see Theorem 3.12 (a) in \cite{nolan2020univariate}), i.e., $a_n\sum_{i\in\mathcal{I}_{n,j}}\xi_i-b_n\xrightarrow{\mathcal{L}}\mathcal{S}(\eta,\delta,1,0),$
we have
$
C_{p,\beta} S_{n,j}(p,\beta) \xrightarrow{\mathcal{L}}\mathcal{S}(\beta/p,1,1,0),
$
which implies that
$
    S_{n,j}(p,\beta) \xrightarrow{\mathcal{L}}\mathcal{S}(\beta/p,1,C_{\beta}(p)^{-1},0).
$
The proof is finished.
\end{proof}

%%%%%%%%%%%%%%%%%%%%%%%%%%%%%%%%%%%%%%%%%%
%%%%%%%%% Proof of Theorem 5 %%%%%%%%%%%%%
%%%%%%%%%%%%%%%%%%%%%%%%%%%%%%%%%%%%%%%%%%

\begin{proof}[Proof of Theorem \ref{th:large-k-2}]
   Note that
we have the following decomposition:
    \begin{align}
\label{eq:large-k-decomp-2}
    &({k_{n}}/{2})^{1/2}(\tilde{\sigma}_{n,j}(p,\beta) - \sigma_{n,t}^p) - \sigma_{n,t}^p\tilde{\Xi}_{n,j}(p,\beta) \nonumber\\
    &=  ({k_{n}}/{2})^{1/2}(\Delta_n^{p/\beta}\sigma_{t_{(n,j)}}^p - \Delta_n^{p/\beta}\sigma_{t}^p) 
    +(\Delta_n^{p/\beta}\sigma_{t_{(n,j)}}^p - \Delta_n^{p/\beta}\sigma_{t}^p) \tilde{\Xi}_{n,j}(p,\beta) + \tilde{R}_{n,j},
\end{align}
where
    \begin{align*}
   \tilde{R}_{n,j} &=(k_{n}/2)^{1/2}(\tilde{\sigma}_{n,j}(p,\beta)-\Delta_n^{p/\beta}\sigma_{t_{(n,j)}}^p ) - \Delta_n^{p/\beta}\sigma_{t_{(n,j)}}^p\tilde{\Xi}_{n,j}(p,\beta). 
  \end{align*}
Similar to \eqref{eq:esti-first} and \eqref{eq:esti-second}, we have
\begin{align}
\label{eq:esti-first-2}
    \mathbb{E}[|(k_{n}/2)^{1/2}(\Delta_n^{p/\beta}\sigma_{t_{(n,j)}}^p - \Delta_n^{p/\beta}\sigma_{t}^p)|] 
    &\leq K \Delta_n^{p\left(\kappa +\frac{1}{\beta}-(\kappa+\frac{1}{2p})\gamma\right)},
\end{align}
and since  $\tilde{\Xi}_{n,j}(p,\beta)=O_{\mathbb{P}}(1)$, we also have
\begin{align}
\label{eq:esti-second-2}
   \mathbb{E}[|(\Delta_n^{p/\beta}\sigma_{t_{(n,j)}}^p - \Delta_n^{p/\beta}\sigma_{t}^p)\tilde{\Xi}_{n,j}(p,\beta)| ] &\leq  K \Delta_n^{p\left(\kappa +\frac{1}{\beta}-\kappa \gamma\right)}.
\end{align}

For $\tilde{R}_{n,j}$, it holds by (\ref{eq:large-k-estimator-2}) and the definition of $\tilde{\Xi}_{n,j}(p,\beta)$ that
 \begin{align*}  
   \tilde{R}_{n,j}
   % &= (k_{n}/2)^{1/2}\left(\frac{1}{\tilde{c}_{\beta}(p)(k_{n}/2)}\sum_{i \in \tilde{\mathcal{I}}_{n,j}} |\Delta_{2i}^n X-\Delta_{2i-1}^n X|^p
   % -\Delta_n^{p/\beta}\sigma_{t_{(n,j)}}^p \right) 
   % \\
   % &\quad- \Delta_n^{p/\beta}\sigma_{t_{(n,j)}}^p
   % \frac{1}{(k_{n}/2)^{1/2}} \sum_{i\in\tilde{\mathcal{I}}_{n,j}}\left(\dfrac{|\Delta_{2i}^n Z-\Delta_{2i-1}^n Z|^p}{\tilde{c}_{\beta}(p)\Delta_n^{p/\beta}}-1\right)
   %  \\
    &= (k_{n}/2)^{-1/2}\tilde{c}_{\beta}(p)^{-1} \sum_{i\in \tilde{\mathcal{I}}_{n,j}} (|\Delta_{2i}^n X-\Delta_{2i-1}^n X|^p-|\sigma_{t_{(n,j)}}(\Delta_{2i}^n Z-\Delta_{2i-1}^n Z)|^p).
\end{align*}
Then we obtain from triangle inequality that for
$0<p<\beta/2<1$,
\begin{align*}
    \mathbb{E}[|\tilde{R}_{n,j}|] &\leq K_{p,\beta}(k_{n}/2)^{-1/2} \sum_{i\in\tilde{\mathcal{I}}_{n,j}} \mathbb{E}[|(\Delta_{2i}^n X-\Delta_{2i-1}^n X)-\sigma_{t_{(n,j)}}(\Delta_{2i}^n Z-\Delta_{2i-1}^n Z)|^p].
\end{align*}
By (\ref{eq:Price-process}) and the triangle inequality, we obtain 
 \begin{align*}
  &\mathbb{E}[|(\Delta_{2i}^n X-\Delta_{2i-1}^n X)-\sigma_{t_{(n,j)}}(\Delta_{2i}^n Z-\Delta_{2i-1}^n Z)|^p]
  \\
  &\leq K_p \mathbb{E}\left[\left|\int_{(2i-1)\Delta_n}^{2i\Delta_n} (b_s-b_{t_{(n,j)}})\,ds\right|^p\right] + K_p 
  \mathbb{E}\left[\left|\int_{(2i-2)\Delta_n}^{(2i-1)\Delta_n} (b_s-b_{t_{(n,j)}})\,ds\right|^p\right]
  \\
  &\quad+ \mathbb{E}\left[\left| \int_{(2i-1)\Delta_n}^{2i\Delta_n} (\sigma_{s-}-\sigma_{t_{(n,j)}})\,dZ_s\right|^p\right] 
  +
  \mathbb{E}\left[\left| \int_{(2i-2)\Delta_n}^{(2i-1)\Delta_n} (\sigma_{s-}-\sigma_{t_{(n,j)}})\,dZ_s\right|^p\right].
 \end{align*}
 
Note that we have for $q>1$,
\begin{align*}
     \mathbb{E}\left[\left|\int_{(2i-1)\Delta_n}^{2i\Delta_n} (b_s-b_{t_{(n,j)}})\,ds\right|^p\right] 
     &\leq K \left(\mathbb{E}\left[\left|\int_{(2i-1)\Delta_n}^{2i\Delta_n} (b_s-b_{t_{(n,j)}})\,ds\right|^q\right]\right)^{p/q}
     \\
     &\leq K \left(\Delta_n^{q-1}\int_{(2i-1)\Delta_n}^{2i\Delta_n} \mathbb{E}\left[\left|b_s-b_{t_{(n,j)}}\right|^q\right]\,ds\right)^{p/q}
     \\
     % &\leq K \Delta_n^p (k_{n}\Delta_n)^{p\tilde{\kappa}}
     % \\
     & \leq  K \Delta_n^{p(\tilde{\kappa}+1-\tilde{\kappa}\gamma)},
\end{align*}
where the first inequality and the second inequality follow by    H\"{o}lder's inequality, the third inequality follows from Assumption \ref{assumption-2} and 
$k_{n}\asymp\Delta_n^{-\gamma}$.
Similarly, we have 
\begin{align*}
     \mathbb{E}\left[\left|\int_{(2i-2)\Delta_n}^{(2i-1)\Delta_n} (b_s-b_{t_{(n,j)}})\,ds\right|^p\right]    
     \leq K \Delta_n^{p(\tilde{\kappa}+1-\tilde{\kappa}\gamma)}.
\end{align*}
On the other hand,
by Lemma \ref{le:Mom-esti} and Assumption \ref{assumption-2}, we obtain
that
\begin{align*}
     \mathbb{E}\left[\left| \int_{(2i-1)\Delta_n}^{2i\Delta_n} (\sigma_{s-}-\sigma_{t_{(n,j)}})\,dZ_s\right|^p\right] 
     &\leq K_{\beta,p}  \left( \int_{(2i-1)\Delta_n}^{2i\Delta_n} \mathbb{E}\left[\left|\sigma_{s-}-\sigma_{t_{(n,j)}}\right|^\beta\right]\,ds\right)^{p/\beta}
     \\
     &\leq K_{\beta,p} (\Delta_n(k_{n}\Delta_n)^{\beta\kappa})^{p/\beta}
     \\
     &\leq K_{\beta,p} \Delta_n^{p(\kappa+\frac{1}{\beta}-\kappa\gamma)},
\end{align*}
and that
\begin{align*}
     \mathbb{E}\left[\left| \int_{(2i-2)\Delta_n}^{(2i-1)\Delta_n} (\sigma_{s-}-\sigma_{t_{(n,j)}})\,dZ_s\right|^p\right] 
     &\leq K_{\beta,p} \Delta_n^{p(\kappa+\frac{1}{\beta}-\kappa\gamma)}.
\end{align*}
By combing the above estimates, we obtain
\begin{align}
\label{eq:tilde-core-inequality-stable}
  \mathbb{E}[|\Delta_{2i}^nX-\Delta_{2i-1}^nX - \sigma_{t_{(n,j)}}(\Delta_{2i}^nZ-\Delta_{2i-1}^nZ)|^p] \leq K_{\beta,p}  \Delta_n^{p((\tilde{\kappa}+1-\tilde{\kappa}\gamma)\wedge (\kappa+\frac{1}{\beta}-\kappa\gamma))}.
\end{align}
And so,
\begin{align}
\label{eq:tilde-R_nj}
 \mathbb{E}[|\tilde{R}_{n,j}|]   
% &\leq K_{\beta,p}  \Delta_n^{p((\tilde{\kappa}+1-\tilde{\kappa}\gamma)\wedge (\kappa+1/\beta-\kappa\gamma))-\gamma/2}
% \nonumber\\
 &\leq  K_{\beta,p}  \Delta_n^{p((\tilde{\kappa}+1-(\tilde{\kappa}+\frac{1}{2p})\gamma)\wedge (\kappa+\frac{1}{\beta}-(\kappa+\frac{1}{2p})\gamma))}.
\end{align}

Therefore,
by combing \eqref{eq:large-k-decomp-2}, \eqref{eq:esti-first-2}, \eqref{eq:esti-second-2} and \eqref{eq:tilde-R_nj}, we obtain
\begin{align*}
  &\mathbb{E} \left[ \left|(k_{n}/2)^{1/2}(\tilde{\sigma}_{n,j}(p,\beta) - \sigma_{n,t}^p) - \sigma_{n,t}^p\tilde{\Xi}_{n,j}(p,\beta)\right|\right]\leq K_{p,\beta} \Delta_n^{p\left(\left(\tilde{\kappa}+1-(\tilde{\kappa}+\frac{1}{2p})\gamma\right)\wedge\left(\kappa+\frac{1}{\beta}-(\kappa+\frac{1}{2p})\gamma\right)\right)}.
\end{align*}
Since $\sigma_{n,t}^p = O_{\mathbb{P}}(\Delta_n^{p/\beta})$, $\beta\in(1,2)$ and $0<\gamma <\frac{\tilde{\kappa}+1-\frac{1}{\beta}}{\tilde{\kappa}+\frac{1}{2p}}\wedge\frac{\kappa}{\kappa+\frac{1}{2p}}$,
we have 
 \begin{align*}
    \dfrac{(k_{n}/2)^{1/2}(\tilde{\sigma}_{n,j}(p,\beta) - \sigma_{n,t}^p)}{\sigma_{n,t}^p} - \tilde{\Xi}_{n,j}(p,\beta) 
    % &= O_{\mathbb{P}}\left(\Delta_n^{p\left(\left(\tilde{\kappa}+1-\frac{1}{\beta}-(\tilde{\kappa}+\frac{1}{2p})\gamma\right)\wedge\left(\kappa-(\kappa+\frac{1}{2p})\gamma\right)\right)}\right) 
    % \\
    &= o_{\mathbb{P}}\left(\Delta_n^{\varepsilon}\right),
\end{align*}
where $0<\varepsilon<p\left(\left(\tilde{\kappa}+1-\frac{1}{\beta}-(\tilde{\kappa}+\frac{1}{2p})\gamma\right)\wedge\left(\kappa-(\kappa+\frac{1}{2p})\gamma\right)\right)$.
The proof for transform $f$ is similar to the proof of in Theorem \ref{th:large-k-1}, and we omit it.
\end{proof}

%%%%%%%%%%%%%%%%%%%%%%%%%%%%%%%%%%%%%%%%%%
%%%%%%%%% Proof of Theorem 6 %%%%%%%%%%%%%
%%%%%%%%%%%%%%%%%%%%%%%%%%%%%%%%%%%%%%%%%%

\begin{proof}[Proof of theorem \ref{th:large-k-2-stable}]
 Note that
we have the following decomposition:
    \begin{align}
\label{eq:large-k-decomp-2-stable}
    &(k_{n}/2)^{1-p/\beta}(\tilde{\sigma}_{n,j}(p,\beta) - \sigma_{n,t}^p) - \sigma_{n,t}^p\tilde{S}_{n,j}(p,\beta) \nonumber\\
    &=  (k_{n}/2)^{1-p/\beta}(\Delta_n^{p/\beta}\sigma_{t_{(n,j)}}^p - \Delta_n^{p/\beta}\sigma_{t}^p) 
   +(\Delta_n^{p/\beta}\sigma_{t_{(n,j)}}^p - \Delta_n^{p/\beta}\sigma_{t}^p)\tilde{S}_{n,j}(p,\beta) + \tilde{R}_{n,j}',
\end{align}
where
    \begin{align*}
   \tilde{ R}_{n,j}' &=(k_{n}/2)^{1-p/\beta}(\tilde{\sigma}_{n,j}(p,\beta)-\Delta_n^{p/\beta}\sigma_{t_{(n,j)}}^p ) - \Delta_n^{p/\beta}\sigma_{t_{(n,j)}}^p\tilde{S}_{n,j}(p,\beta). 
  \end{align*}

For $p\in(0,1\wedge\beta)$, it holds by the inequality $|a^p-b^p|\leq K_p|a-b|^p$ and Assumption \ref{assumption-2} that
\begin{align}
\label{eq:tilde-esti-first-stable}
    \mathbb{E}[|(k_{n}/2)^{1-p/\beta}(\Delta_n^{p/\beta}\sigma_{t_{(n,j)}}^p - \Delta_n^{p/\beta}\sigma_{t}^p)|] 
   %  &\leq K_{p,\beta}k_{n}^{1-p/\beta}\Delta_n^{p/\beta}(k_{n}\Delta_n)^{p\kappa}
   % \nonumber \\
    &\leq K_{p,\beta} \Delta_n^{p\left(\kappa +\frac{1}{\beta}-(\kappa+\frac{1}{p}-\frac{1}{\beta})\gamma\right)}.
\end{align}
Since  $\tilde{S}_{n,j}(p,\beta)=O_{\mathbb{P}}(1)$, we also have
\begin{align}
\label{eq:tilde-esti-second-stable}
   \mathbb{E}[|(\Delta_n^{p/\beta}\sigma_{t_{(n,j)}}^p - \Delta_n^{p/\beta}\sigma_{t}^p)\tilde{S}_{n,j}(p,\beta)| ] &\leq  K \Delta_n^{p\left(\kappa +\frac{1}{\beta}-\kappa \gamma\right)}.
\end{align}
For $\tilde{R}_{n,j}'$, it holds by (\ref{eq:large-k-estimator-2}) and
the definition of $\tilde{S}_{n,j}(p,\beta)$ that
 \begin{align*}  
   \tilde{R}'_{n,j}
   % &= (k_{n}/2)^{1-p/\beta}\left(\frac{1}{\tilde{c}_{\beta}(p)(k_{n}/2)}\sum_{i \in \tilde{\mathcal{I}}_{n,j}} |\Delta_{2i}^n X-\Delta_{2i-1}^n X|^p
   % -\Delta_n^{p/\beta}\sigma_{t_{(n,j)}}^p \right) 
   % \\
   % &\quad- \Delta_n^{p/\beta}\sigma_{t_{(n,j)}}^p
   % \frac{1}{(k_{n}/2)^{p/\beta}} \sum_{i\in\tilde{\mathcal{I}}_{n,j}}\left(\dfrac{|\Delta_{2i}^n Z-\Delta_{2i-1}^n Z|^p}{\tilde{c}_{\beta}(p)\Delta_n^{p/\beta}}-1\right)
   %  \\
    &= (k_{n}/2)^{-p/\beta}\tilde{c}_{\beta}(p)^{-1} \sum_{i\in \tilde{\mathcal{I}}_{n,j}} (|\Delta_{2i}^n X-\Delta_{2i-1}^n X|^p-|\sigma_{t_{(n,j)}}(\Delta_{2i}^n Z-\Delta_{2i-1}^n Z)|^p).
\end{align*}
From \eqref{eq:tilde-core-inequality-stable}, we have 
\begin{align}
\label{eq:tilde-R_nj-stable}
 \mathbb{E}[|\tilde{R}'_{n,j}|]   
% &\leq K_{\beta,p}  \Delta_n^{p((\tilde{\kappa}+1-\tilde{\kappa}\gamma)\wedge (\kappa+1/\beta-\kappa\gamma))-(1-p/\beta)\gamma}
% \nonumber\\
 &\leq K_{\beta,p}  \Delta_n^{p((\tilde{\kappa}+1-(\tilde{\kappa}+\frac{1}{p}-\frac{1}{\beta})\gamma)\wedge (\kappa+\frac{1}{\beta}-(\kappa+\frac{1}{p}-\frac{1}{\beta})\gamma))}. 
\end{align}
Hence, by combing \eqref{eq:large-k-decomp-2-stable},
\eqref{eq:tilde-esti-first-stable},
\eqref{eq:tilde-esti-second-stable} and \eqref{eq:tilde-R_nj-stable}, we obtain
\begin{align*}
    &\mathbb{E}[|(k_{n}/2)^{1-p/\beta}(\tilde{\sigma}_{n,j}(p,\beta) - \sigma_{n,t}^p) - \sigma_{n,t}^p\tilde{S}_{n,j}(p,\beta)|]
    \\
    &\leq  K_{\beta,p}  \Delta_n^{p((\tilde{\kappa}+1-(\tilde{\kappa}+\frac{1}{p}-\frac{1}{\beta})\gamma)\wedge (\kappa+\frac{1}{\beta}-(\kappa+\frac{1}{p}-\frac{1}{\beta})\gamma))}.
\end{align*}
Since $\sigma_{n,t} = O_{\mathbb{P}}(\Delta_n^{1/\beta})$, $\beta\in(1/(1+\tilde{\kappa}),2)$ and $0<\gamma <\frac{\tilde{\kappa}+1-\frac{1}{\beta}}{\tilde{\kappa}+\frac{1}{p}-\frac{1}{\beta}}\wedge\frac{\kappa}{\kappa+\frac{1}{p}-\frac{1}{\beta}}$, we have 
 \begin{align*}
    \dfrac{(k_{n}/2)^{1-p/\beta}(\tilde{\sigma}_{n,j}(p,\beta) - \sigma_{n,t}^p)}{\sigma_{n,t}^p} -\tilde{S}_{n,j}(p,\beta)
    % \\
    % &= O_{\mathbb{P}}\left(\Delta_n^{p((\tilde{\kappa}+1-\frac{1}{\beta}-(\tilde{\kappa}+\frac{1}{p}-\frac{1}{\beta})\gamma)\wedge (\kappa-(\kappa+\frac{1}{p}-\frac{1}{\beta})\gamma))}\right) 
    % \\
    &= o_{\mathbb{P}}\left(\Delta_n^{\varepsilon}\right),
\end{align*}
where $0<\varepsilon<p((\tilde{\kappa}+1-\frac{1}{\beta}-(\tilde{\kappa}+\frac{1}{p}-\frac{1}{\beta})\gamma)\wedge (\kappa-(\kappa+\frac{1}{p}-\frac{1}{\beta})\gamma))$.

For $p\in [1,\beta)$,
similar to \eqref{eq:esti-first-stable-big-1}
and \eqref{eq:esti-second-stable-big-1}, we obtain 
\begin{align}
\label{eq:tilde-esti-first-stable-big-1}
    \mathbb{E}[|(k_{n}/2)^{1-p/\beta}(\Delta_n^{p/\beta}\sigma_{t_{(n,j)}}^p - \Delta_n^{p/\beta}\sigma_{t}^p)|] 
   %  &\leq K_{p,\beta}k_{n}^{1-p/\beta}\Delta_n^{p/\beta}(k_{n}\Delta_n)^{\kappa}
   % \nonumber \\
    &\leq K_{p,\beta} \Delta_n^{\kappa +\frac{p}{\beta}-(\kappa+1-\frac{p}{\beta})\gamma},
\end{align}
and 
\begin{align}
\label{eq:tilde-esti-second-stable-big-1}
   \mathbb{E}[|(\Delta_n^{p/\beta}\sigma_{t_{(n,j)}}^p - \Delta_n^{p/\beta}\sigma_{t}^p)\tilde{S}_{n,j}'(p,\beta)| ] &\leq  K \Delta_n^{\kappa +\frac{p}{\beta}-\kappa \gamma}.
\end{align}
For $\tilde{R}_{n,j}'$, we have
\begin{align*}
   & \mathbb{E}[||\Delta_{2i}^n X-\Delta_{2i-1}^n X|^p-|\sigma_{t_{(n,j)}}(\Delta_{2i}^n Z-\Delta_{2i-1}^n Z)|^p|]
   % \\
   % & \leq  K_p \mathbb{E}[(|\Delta_i^n X-\Delta_{i-2}^n X|^{p-1}+|\sigma_{t_{(n,j)}}(\Delta_i^n Z-\Delta_{i-2}^n Z)|^{p-1})
   % \\
   % &\quad\times |(\Delta_i^n X-\Delta_{i-2}^n X)-\sigma_{t_{(n,j)}}(\Delta_i^n Z-\Delta_{i-2}^n Z)|]
   \\
   &\leq  K_p (\mathbb{E}[|\Delta_{2i}^n X-\Delta_{2i-1}^n X|^p+|\Delta_{2i}^n Z-\Delta_{2i-1}^n Z|^p])^{\frac{p-1}{p}}
   \\
   &\quad\times (\mathbb{E}[|(\Delta_{2i}^n X-\Delta_{2i-1}^n X)-\sigma_{t_{(n,j)}}(\Delta_{2i}^n Z-\Delta_{2i-1}^n Z)|^p])^{1/p}
   % \\
   % &\leq  K_{\beta,p} \Delta_n^{\frac{p-1}{\beta}}  \Delta_n^{(\tilde{\kappa}+1-\tilde{\kappa}\gamma)\wedge (\kappa+1/\beta-\kappa\gamma)}
   \\
   &\leq K_{\beta,p}   \Delta_n^{(\tilde{\kappa}+1+\frac{p-1}{\beta}-\tilde{\kappa}\gamma)\wedge (\kappa+\frac{p}{\beta}-\kappa\gamma)},
\end{align*}
where we used $|\Delta_{2i}^nX-\Delta_{2i-1}^nX|=O_{\mathbb{P}}(\Delta_n^{1/\beta})$, $|\Delta_{2i}^nZ-\Delta_{2i-1}^nZ|=O_{\mathbb{P}}(\Delta_n^{1/\beta})$ and \eqref{eq:tilde-core-inequality-stable}
in the second inequality. Subsequently, we obtain
\begin{align}
\label{eq:tilde-R_nj-stable-big-1}
 \mathbb{E}[| \tilde{R}_{n,j}' |] 
 % &\leq  K_{\beta,p}k_{n}^{1-p/\beta}   \Delta_n^{(\tilde{\kappa}+1+\frac{p-1}{\beta}-\tilde{\kappa}\gamma)\wedge (\kappa+\frac{p}{\beta}-\kappa\gamma)}
 % \nonumber\\
 &\leq K_{\beta,p}  \Delta_n^{(\tilde{\kappa}+1+\frac{p-1}{\beta}-(\tilde{\kappa}+1-\frac{p}{\beta})\gamma)\wedge (\kappa+\frac{p}{\beta}-(\kappa+1-\frac{p}{\beta})\gamma)}. 
\end{align}
Hence, by combing \eqref{eq:large-k-decomp-2-stable},
\eqref{eq:tilde-esti-first-stable-big-1}, \eqref{eq:tilde-esti-second-stable-big-1} and
\eqref{eq:tilde-R_nj-stable-big-1}, we obtain
\begin{align*}
 &\mathbb{E}[|(k_{n}/2)^{1-p/\beta}(\tilde{\sigma}_{n,j}(p,\beta) - \sigma_{n,t}^p) - \sigma_{n,t}^p\tilde{S}_{n,j}(p,\beta)|] 
 \\
 &\leq K_{\beta,p}  \Delta_n^{(\tilde{\kappa}+1+\frac{p-1}{\beta}-(\tilde{\kappa}+1-\frac{p}{\beta})\gamma)\wedge (\kappa+\frac{p}{\beta}-(\kappa+1-\frac{p}{\beta})\gamma)}.  
\end{align*}
Since $\sigma_{n,t}^p = O_{\mathbb{P}}(\Delta_n^{p/\beta})$, 
$\beta\in(1/(1+\tilde{\kappa}),2)$, $p\in [1,\beta)$ and $0<\gamma <\frac{\kappa}{\kappa+1-\frac{p}{\beta}}=\frac{\kappa}{\kappa+1-\frac{p}{\beta}}\wedge\frac{\tilde{\kappa}+1-\frac{1}{\beta}}{\tilde{\kappa}+1-\frac{p}{\beta}}$, we have 
 \begin{align*}
    \dfrac{(k_{n}/2)^{1-p/\beta}(\tilde{\sigma}_{n,j}(p,\beta) - \sigma_{n,t}^p)}{\sigma_{n,t}^p} -\tilde{S}_{n,j}(p,\beta) 
    % &= O_{\mathbb{P}}\left(\Delta_n^{(\tilde{\kappa}+1-\frac{1}{\beta}-(\tilde{\kappa}+1-\frac{p}{\beta})\gamma)\wedge (\kappa-(\kappa+1-\frac{p}{\beta})\gamma)}\right) 
    % \\
    &= o_{\mathbb{P}}\left(\Delta_n^{\varepsilon}\right),
\end{align*}
where $0<\varepsilon<{(\tilde{\kappa}+1-\frac{1}{\beta}-(\tilde{\kappa}+1-\frac{p}{\beta})\gamma)\wedge (\kappa-(\kappa+1-\frac{p}{\beta})\gamma)}$.
The proof for transform $f$ is similar to the proof of in Theorem \ref{th:large-k-1}, and we omit it.
\end{proof}

%%%%%%%%%%%%%%%%%%%%%%%%%%%%%%%%%%%%%%%%%%
%%%%%%%%% Proof of Lemma 3 %%%%%%%%%%%%%%%
%%%%%%%%%%%%%%%%%%%%%%%%%%%%%%%%%%%%%%%%%%

\begin{proof}[Proof of Lemma \ref{le:lim-distri-tilde-S_nj}]
By the self-similarity property of the stable process, we have
\begin{align*}
   \tilde{S}_{n,j}(p,\beta)
   \;\overset{\mathcal{L}}{=}\;
   \frac{1}{(k_{n}/2)^{p/\beta}}
   \sum_{i\in\tilde{\mathcal{I}}_{n,j}}
   \left(\frac{|\tilde{Z}_i|^p}{\tilde{c}_{\beta}(p)}-1\right),
\end{align*}
where $(\tilde{Z}_i)_{i\in\tilde{\mathcal{I}}_{n,j}}$ are i.i.d. stable random variables satisfying $\mathbb{E}[|\tilde{Z}_i|^p]=\tilde{c}_\beta(p)$ and
$\tilde{Z}_i\sim\mathcal{S}(\beta,0,1,0)$.
For any $i\in\tilde{\mathcal{I}}_{n,j}$, let us set 
$\tilde{\xi}_i \equiv |\tilde{Z}_i|^p / \tilde{c}_\beta(p)$.
Then $(\tilde{\xi}_i)_{i\in\tilde{\mathcal{I}}_{n,j}}$ are i.i.d.\ variables with $\mathbb{E}[\tilde{\xi}_i]=1$. 
Note that
\begin{align*}
\lim_{x\to\infty} x^{\beta/p}\mathbb{P}[\tilde{\xi}_i > x]
&=\lim_{x\to\infty} x^{\beta/p}\mathbb{P}[|\tilde{Z}_i| > (x\tilde{c}_\beta(p))^{1/p}]=\frac{2\Gamma(\beta)\sin(\pi\beta/2)}{\pi}\,
\tilde{c}_\beta(p)^{-\beta/p},
\end{align*}
where the second equality follows from the fact $\tilde{Z}_i\sim\mathcal{S}(\beta,0,1,0)$ and Theorem~1.2 in \cite{nolan2020univariate}.
% Furthermore, it is clear that 
% $\lim_{x\to\infty}x^{\beta/p}\mathbb{P}[\tilde{\xi}_i<-x]=0$. 

By taking $\eta=\beta/p$, $c^-=0$, and 
$c^+=2\tilde{c}_\beta(p)^{-\beta/p}\Gamma(\beta)\sin(\pi\beta/2)/\pi$, we have
\begin{align*}
\delta &\equiv \frac{c^+ - c^-}{c^+ + c^-} = 1,
\\
a_n &\equiv (k_{n}/2)^{-1/\eta}
\left(
\frac{2\Gamma(\eta)\sin(\pi\eta/2)}{\pi(c^+ + c^-)}
\right)^{1/\eta}
=\tilde{C}_\beta(p)(k_{n}/2)^{-p/\beta},
\\
b_n &\equiv (k_{n}/2)\tilde{C}_\beta(p)a_n\mathbb{E}[\xi]
= (k_{n}/2)\tilde{C}_\beta(p)k_{n}^{-p/\beta},
\end{align*}
% where
% \[
% \tilde{C}_\beta(p)
% = \tilde{c}_\beta(p)
% \left(
% \frac{\Gamma(\beta/p)\sin(\pi\beta/(2p))}
% {\Gamma(\beta)\sin(\pi\beta/2)}
% \right)^{p/\beta}.
% \]
Hence, by the generalized CLT (see Theorem 3.12 (a) in \cite{nolan2020univariate}),
we have 
\(
\tilde{C}_{p,\beta}\tilde{S}_{n,j}(p,\beta)
\;\xrightarrow{\mathcal{L}}\;
\mathcal{S}\!\left({\beta}/{p},1,1,0\right),
\)
which implies that
\(
\tilde{S}_{n,j}(p,\beta)
\;\xrightarrow{\mathcal{L}}\;
\mathcal{S}\!\left({\beta}/{p},1,\tilde{C}_\beta(p)^{-1},0\right).
\)
The proof is completed.
\end{proof}

%%%%%%%%%%%%%%%%%%%%%%%%%%%%%%%%%%%%%%%%%%%%%%%
%%% Example with single Appendix:            %%
%%%%%%%%%%%%%%%%%%%%%%%%%%%%%%%%%%%%%%%%%%%%%%%
%\begin{appendix}
%\section*{Title}\label{appn} %% if no title is needed, leave empty \section*{}.
%Appendices should be provided in \verb|{appendix}| environment,
%before Acknowledgements.
%
%If there is only one appendix,
%then please refer to it in text as \ldots\ in the \hyperref[appn]{Appendix}.
%\end{appendix}
%%%%%%%%%%%%%%%%%%%%%%%%%%%%%%%%%%%%%%%%%%%%%%%
%%% Example with multiple Appendixes:        %%
%%%%%%%%%%%%%%%%%%%%%%%%%%%%%%%%%%%%%%%%%%%%%%%
% \begin{appendix}
% \section{Other  results}\label{appA}
%Sample of cross-reference to the formula (\ref{path}) in Appendix \ref{appB}.
% \end{appendix}

%%%%%%%%%%%%%%%%%%%%%%%%%%%%%%%%%%%%%%%%%%%%%%
%% Support information, if any,             %%
%% should be provided in the                %%
%% Acknowledgements section.                %%
%%%%%%%%%%%%%%%%%%%%%%%%%%%%%%%%%%%%%%%%%%%%%%
%\begin{acks}[Acknowledgments]
%The authors would like to thank the anonymous referees, an Associate
%Editor and the Editor for their constructive comments that improved the
%quality of this paper.
%\end{acks}

% \newpage

%%%%%%%%%%%%%%%%%%%%%%%%%%%%%%%%%%%%%%%%%%%%%%
%% Funding information, if any,             %%
%% should be provided in the                %%
%% funding section.                         %%
%%%%%%%%%%%%%%%%%%%%%%%%%%%%%%%%%%%%%%%%%%%%%%
\begin{funding}
D. Chen's research is supported by Singapore Ministry of Education Tier 1 Grant (Grant ID: 24-SOE-SMU-042), and National Natural Science Foundation of China (Grant ID: 12571515).
J. Li's research is supported by Singapore Ministry of Education Tier 2 Grant (Grant ID: T2EP40124-0021).
\end{funding}

\bibliographystyle{imsart-number} % Style BST file (imsart-number.bst or imsart-nameyear.bst)
\bibliography{reference}      % Bibliography file (usually '*.bib')

%% or include bibliography directly:
%\begin{thebibliography}{4}
%%%
%\bibitem{r1}
%\textsc{Billingsley, P.} (1999). \textit{Convergence of
%Probability Measures}, 2nd ed.
%Wiley, New York.
%
%\bibitem{r2}
%\textsc{Bourbaki, N.}  (1966). \textit{General Topology}  \textbf{1}.
%Addison--Wesley, Reading, MA.
%
%\bibitem{r3}
%\textsc{Ethier, S. N.} and \textsc{Kurtz, T. G.} (1985).
%\textit{Markov Processes: Characterization and Convergence}.
%Wiley, New York.
%
%\bibitem{r4}
%\textsc{Prokhorov, Yu.} (1956).
%Convergence of random processes and limit theorems in probability
%theory. \textit{Theory  Probab.  Appl.}
%\textbf{1} 157--214.
%\end{thebibliography}

\end{document}